\documentclass[a4paper, 11pt]{article}

\usepackage[english]{babel}
\usepackage[utf8x]{inputenc}
\usepackage[T1]{fontenc}
\usepackage{amsmath}
\usepackage{amssymb}
\usepackage{amsthm}
\usepackage{stmaryrd}
\usepackage[pagewise, displaymath, mathlines]{lineno}

\usepackage[a4paper,top=3cm,bottom=2cm,left=2.5cm,right=2.5cm,marginparwidth=1.75cm]{geometry}

\usepackage{enumerate}
\usepackage{amsmath}
\usepackage{graphicx}
\usepackage[colorinlistoftodos]{todonotes}
\usepackage[colorlinks=true, allcolors=blue]{hyperref}
\usepackage{mathtools}

\newtheorem{thm}{Theorem}
\newtheorem{defi}[thm]{Definition}
\newtheorem{prop}[thm]{Proposition}
\newtheorem{coroll}[thm]{Corollary}
\newtheorem{rmk}[thm]{Remark}
\newtheorem{lemma}[thm]{Lemma}
\newtheorem{hypo}{Hypothesis}

\newcommand\blfootnote[1]{%
  \begingroup
  \renewcommand\thefootnote{}\footnote{#1}%
  \addtocounter{footnote}{-1}%
  \endgroup
}

\newcommand{\RR}{\mathbb{R}}
\newcommand{\pO}{\partial \Omega}

\newcommand{\EE}{\mathbb{E}}

\newcommand{\supp}{\text{supp}}

\newcommand{\vertiii}[1]{{\left\vert\kern-0.25ex\left\vert\kern-0.25ex\left\vert #1 
    \right\vert\kern-0.25ex\right\vert\kern-0.25ex\right\vert}}

\newcommand{\vp}{v_{\perp}}
\newcommand{\vt}{v_{\parallel}}
 \newcommand{\up}{u_{\perp}}
 \newcommand{\ut}{u_{\parallel}}

   \newcommand{\rp}{r_{\perp}}
 \newcommand{\rt}{r_{\parallel}}

 \newcommand{\e}{\epsilon}
 
 \renewcommand{\d}{\mathrm{d}}

\title{Convergence Towards the Steady State of a Collisionless Gas With Cercignani-Lampis Boundary Condition.}
\author{Armand Bernou\footnote{Sorbonne Universit\'e and Universit\'e de Paris, CNRS, Laboratoire Jacques-Louis Lions (LJLL), F-75005 Paris, France. E-mail address: armand.bernou@sorbonne-universite.fr.} }

\begin{document}

\maketitle

\abstract{
We study the asymptotic behavior of the kinetic free-transport equation enclosed in a regular domain, on which no symmetry assumption is made, with Cercignani-Lampis boundary condition. We give the first proof of existence of a steady state in the case where the temperature at the wall varies, and derive the optimal rate of convergence towards it, in the $L^1$ norm. The strategy is an application of a deterministic version of Harris' subgeometric theorem, in the spirit of \cite{Canizo_Mischler_2020} and \cite{Bernou_Transport_Semigroup_2020}. We also investigate rigorously the velocity flow of a model mixing pure diffuse and Cercignani-Lampis boundary conditions with variable temperature, for which we derive an explicit form for the steady state, providing new insights on the role of the Cercignani-Lampis boundary condition in this problem.
}



\textbf{Acknowledgements:} The author acknowledges financial support from the European Research Council (ERC) under the European Union’s Horizon 2020 research and innovation programme (Grant Agreement number 864066).

\tableofcontents

\section{Introduction}

\subsection{Model and boundary condition}
In this paper, we consider the kinetic free-transport equation with Cercignani-Lampis boundary condition, inside a bounded domain (open, connected) $\Omega \subset \RR^d$, $d \in \{2,3\}$. The corresponding initial boundary value problem writes
\begin{align}
\label{eq:main_pb}
\left\{
\begin{array}{lll}
\partial_t f(t,x,v) + v \cdot \nabla_x f(t,x,v) = 0, \qquad \qquad &(t,x,v) \in (0,\infty) \times G, \\
\gamma_- f(t,x,v) = K \gamma_+ f(t,x,v), \quad & (t,x,v) \in \RR_+ \times \Sigma_-, \\
f(0,x,v) = f_0(x,v), & (x,v) \in G,
\end{array}
\right.
\end{align}
with the notations $G := \Omega \times \RR^d$, and, denoting $n_x$ the unit \textbf{outward} normal vector at $x \in \pO$, 
\begin{align*}
\Sigma := \pO \times \RR^d, \qquad \Sigma_{\pm} := \big\{(x,v) \in \Sigma, \pm (v \cdot n_x) > 0 \big \}.
\end{align*}
Let us introduce the boundary operator $K$ corresponding to the Cercignani-Lampis boundary condition. Let $\bar G$ denote the closure of $G$. For a function $\phi$ on $(0,\infty) \times \bar{G}$, we denote $\gamma_{\pm} \phi$ its trace on $(0,\infty) \times \Sigma_{\pm}$, under the assumption that this object is well-defined. The boundary operator $K$ is defined, for $\phi$ supported on $(0,\infty) \times \Sigma_+$, for $(t,x,v) \in (0,\infty) \times \Sigma_-$ and assuming that $\phi(t,x,\cdot) \in L^1(\{v' \in \RR^d : v' \cdot n_x > 0\})$, by
\begin{align}
\label{eq:def_K}
K \phi(t,x,v) = \int_{\Sigma_+^x} \phi(t,x,u) \, R(u \to v;x) \, |u \cdot n_x| \, \d u,
\end{align}
where, for all $x \in \pO$,
\[ \Sigma_{\pm}^x := \big\{v \in \RR^d, (x,v) \in \Sigma_{\pm} \big\}, \]
and where the kernel $R(u \to v;x)$ is given, for $x \in \pO$, $u \in \Sigma_+^x$, $v \in \Sigma_-^x$, by the following formula
\begin{align}
\label{eq:def_R}
R(u \to v; x) &:= \frac{1}{\theta(x) \rp} \frac{1}{(2 \pi \theta(x) \rt (2 - \rt))^{\frac{d-1}{2}}} e^{-\frac{|\vp|^2}{2 \theta(x) \rp}} e^{-\frac{(1-\rp) |\up|^2}{2 \theta(x) \rp}} \\
&\qquad \times e^{-\frac{|\vt - (1-\rt) \ut|^2}{2 \theta(x) \rt (2 - \rt)}} I_0 \Big( \frac{(1-\rp)^{\frac12} \up \cdot \vp}{\theta(x) \rp} \Big), \nonumber
\end{align}
with the following notations:
\begin{align*}
\vp := (v \cdot n_x) n_x, \quad \vt := v - \vp, \quad \up := (u \cdot n_x) n_x, \quad \ut = u - \up,
\end{align*}
where $I_0$ is the modified Bessel function given, for all $y \in \RR$, by
\begin{align}
\label{eq:defI0}
I_0(y) := \frac{1}{\pi} \int_0^{\pi} e^{y \cos \phi} \, \d \phi,
\end{align}
and where $\theta(x) > 0$ is the wall temperature at $x \in \pO$. The coefficients $\rp \in (0,1)$ and $\rt \in (0,2)$ are the two accommodation coefficients (normal and tangential) at the wall. The value $\vp$ is the normal component of the velocity $v$ at the boundary, while $\vt$ is the tangential component. The same interpretation is of course valid for $u$. 

We will heavily use the normalization property, see \cite[Lemma 10]{Chen_CL_2020}, which, with our notation for $R$, writes, for all $(x,u) \in \Sigma_+$,
\begin{align}
\label{eq:normalization_basic}
\int_{\Sigma_-^x} R(u \to v; x) \, |v \cdot n_x| \, \d v = 1. 
\end{align}

This condition ensures the conservation of mass, and the $L^1$ contraction of the semigroup associated to \eqref{eq:main_pb}, see Section \ref{section:setting}.

\subsection{Physical motivations}

In kinetic theory, the free-transport equation with boundary condition models the evolution of a Knudsen (collisionless) gas enclosed in the vessel $\Omega$, and was first examined in the seminal work of Bardos \cite{Bardos_1970}. In this case, the gas is strongly diluted, hence the Lebesgue measure of the set of collisions between particles is $0$ and the collision operator of the Boltzmann equation describing statistically the dynamics vanishes. Gas molecules in $\Omega$ move according to the free-transport dynamics until they meet with the boundary. 

Several models for the description of the reflection at the boundary $\pO$ exist: the simplest choices are the bounce-back boundary condition
\[ f(t,x,v) = f(t,x,-v), \qquad (t,x,v) \in (0,\infty) \times \Sigma_-, \]
and the pure specular boundary condition
\[f(t,x,v) = f(t,x,v - 2(v \cdot n_x) n_x), \qquad (t,x,v) \in (0,\infty) \times \Sigma_-, \]
which are deterministic. Those conditions are unable to render the stress exerted by the gas on the wall, and for this reason, Maxwell \cite[Appendix]{Maxwell_Rarified_Gases_1879} introduced the pure diffuse reflection, in which the particle is adsorbed by the boundary before being re-emitted inside the domain according to a new velocity distribution, defined through some kernel $M$:
\[ f(t,x,v) = M(x,v) \Big(\int_{\Sigma_+^x} f(t,x,u) \, |u \cdot n_x| \, \d u \Big), \qquad (t,x,v) \in (0,\infty) \times \Sigma_-. \] 
The paradigmatic example of such $M$ is the wall Maxwellian
\[ M(x,v) = c(x) e^{-\frac{|v|^2}{2 \theta(x)}}, \]
with $c(x)$ a normalizing constant. In the case of the pure diffuse reflection, there is no correlation between the incoming velocity and the emerging one, for both the normal and the tangential components. A first answer to this issue was the introduction of the so-called Maxwell boundary condition, based on a convex combination between the pure diffuse reflection and the pure specular reflection. 

A more delicate way to address this question, while still modeling the stress exerted by the gas on the boundary, is to consider that the probability distribution appearing in the diffuse reflection retains some information from the impinging velocity. Based on this idea, Cercignani and Lampis \cite{Cercignani_Lampis_1971} introduced what is now known as the Cercignani-Lampis boundary condition, corresponding to the kernel $R$ given by \eqref{eq:def_R}, see also the monograph of Cercignani, Illner and Pulvirenti \cite{Cercignani_Illner_Pulvirenti2004}. In this kernel, two accommodation coefficients are given: one for the normal component, $\rp$, and one for the tangential component $\rt$. This description generalizes that of the diffuse reflection: for $\rt = \rp = 1$, we recover the case of the Maxwellian distribution at the wall mentioned above. As for the specular reflection, it can be considered as a limiting case in which $\rt = \rp = 0$, while the bounce-back boundary condition corresponds to a limiting case with $\rt = 2$ and $\rp = 0$. Some graphs of the distribution induced by the Cercignani-Lampis boundary condition with different sets of accommodation coefficients are provided in Chen \cite[Figures 1-4]{Chen_CL_2020}. 

Already in the 1980's, physical computations showed that, for some models, the Cercignani-Lampis boundary condition provides a more accurate description of the system in comparison with the pre-existing boundary conditions. A particularly interesting case is the computation of the Poiseuille flow and the thermal creep through a tube in the free-molecular regime, see Sharipov \cite{Sharipov_CL_2002} and the references within, in particular \cite{Alexandrychev_CL_1986, Markelov_CL_1982}. The Cercignani-Lampis boundary condition also describes more accurately the behavior, observed experimentally, of a gas nitrogen flow, mainly because of the introduction of the tangential accommodation coefficient which is found slightly different from one, see Pantazis et al. \cite[Sections 3 and 4]{Pantazis_CL_2011}. 

\subsection{Qualitative convergence towards the steady state}

For the free-transport equation considered in this paper, a first key question regarding the asymptotic behavior is whether a steady state exists. While the answer is trivial in the case of the Maxwell boundary condition with constant temperature, it is significantly more involved in the case where the temperature varies, although an explicit form was derived by Sone \cite[Chapter 2, Section 2.5, Equation (2.48)]{Sone_Molecular_Gas_2007}. It is unclear whether such an explicit expression exists for the Cercignani-Lampis boundary condition with varying temperature (one should expect a quite complicated form if that is the case), although some stability properties for Maxwell distributions interacting with this kernel exist, see Lord \cite{Lord_CLL_1991}. On the other hand, it can be easily deduced from \cite[Equation (6)] {Cercignani_Lampis_1971} that an explicit steady state exists in the form of a Maxwellian distribution in the case where the temperature and the accommodation coefficients are constant. Let us mention that for the particular case where the rarefied gas is confined between two parallel plates with varying temperature, a numerical derivation has been obtained by means of an integral equation by Kosuge et al. \cite{Kosuge_steady_2011}. We present in Section \ref{Section:steady-flow} a similar toy model, in which we impose that $\rp = \rt = 1$ and $\theta \equiv 1$ on one of the plate. For this case, we provide an explicit steady state even when the temperature (on the second plate) is allowed to vary. 

Recently, a striking work of Lods, Mokhtar-Kharroubi and Rudnicki \cite{Lods_2020} focusing on the free-transport equation enclosed in a domain with general boundary conditions gives a proof of existence of a steady state for a large class of diffuse, regular (in their terminology) boundary operators. This work was completed by Lods and Mokhtar-Kharroubi in \cite{Lods_2020_Quantitative} by a derivation of some rate of convergence towards this steady state by means of a Tauberian approach. However, the Cercignani-Lampis boundary condition fails to satisfy the ``regular'' property required in those two papers, see Proposition \ref{prop:not_regular}. In this paper, we obtain the existence and uniqueness of the steady state from our results on the convergence, providing the first proof of existence of this steady state when the temperature at the boundary is allowed to vary. 

\subsection{Convergence rate towards the steady state for linear kinetic equations with boundary conditions}

In the present investigation, we are mainly interested in the quantitative study of the convergence towards the steady state. Those questions of quantitative convergence of linear kinetic equations have drawn major interest in the mathematical community during the last decade.

Let us also mention briefly the numerous studies focusing on equations from collisional kinetic theory linearized around an equilibrium, in the $L^2$ setting, with general Maxwell boundary conditions (note that, in this case, we expect convergence towards equilibrium even with the pure specular boundary condition). In particular, we quote here the various applications of the $L^2-L^{\infty}$ theory of Guo, first applied to the Boltzmann equation \cite{guo_decay_2010}, see also Briant-Guo \cite{Briant_Maxwell_2016}, and to the Landau equation with the specular reflection boundary condition, see \cite{Guo_Landau_Specular_L2_2020,Guo_Landau_Specular_2020}. On this matter, we mention also \cite{Kim2017, Kim2018, Duan2020}. A more recent result of Bernou, Carrapatoso, Mischler and Tristani \cite{BCMT_Linear_Maxwell_2020} handles the whole general Maxwell boundary condition for the linearized Boltzmann equation with and without cut-off and the linearized Landau equation based on an adaptation of the (constructive) hypocoercivity method for linear equations developed by Dolbeault-Mouhot-Schmeiser \cite{Dolbeault_note_2009, Dolbeault_Hypocoercivity_2015}. Those $L^2$ methods can not be adapted in a straightforward manner to the Cercignani-Lampis boundary condition, because, as noticed by Chen \cite[Remark 3]{Chen_CL_2020}, it is not possible to view the boundary condition as a projection to obtain the $L^2$ inequality heavily required in the case of the Maxwell boundary condition. New ideas are needed to adapt the hypocoercivity framework to this model. Very recent results of well-posedness have been obtained by Chen \cite{Chen_CL_2020} and, in the convex setting, by Chen, Kim and Li \cite{Chen_Al_2020}.

 For the free-transport equation considered here, with pure diffuse boundary condition, a numerical investigation was first performed by Tsuji, Aoki and Golse \cite{tsuji_relaxation_2010}. In their paper, the rate of convergence, in the $L^1$ norm, was identified as a polynomial rate of order $\frac{1}{t^d}$.  A first analytic study of the model followed, in which Aoki and Golse \cite{Aoki_golse_2011} derived an upper bound of $\frac{1}{t}$ for the convergence in $L^1$ norm, with strong symmetry hypotheses (radial symmetry of the initial data and of the space domain). In a series of articles, Kuo, Liu and Tsai \cite{kuo_free_2013, kuo_liu_tsai_2014} and Kuo \cite{kuo_equilibrating_2015} found the optimal rate $\frac{1}{t^d}$ with the same assumption of radial symmetry of the domain, by using probabilistic arguments, in particular deriving a law of large numbers for the interval of times between two collisions of a particle with the boundary. Ultimately their results allow one to handle the Maxwell boundary condition with various temperatures at the boundary. Another probabilistic approach was taken by Bernou and Fournier \cite{bernou_fournier_collisionless_2019} through the use of a probabilistic coupling, based on a description of the problem with a stochastic process. This allowed the authors to conclude to the optimal rate $\frac{1}{t^d}$ in the general case of a $C^2$ regular domain, with constant temperature. The paper also extends slightly beyond the Maxwellian case by considering other possibilities for $M$ and modifying the rate of convergence accordingly. Some related numerical results are provided in Bernou \cite[Chapter 3]{Bernou_PhD}. Still for the free-transport equation with Maxwell boundary condition, Bernou \cite{Bernou_Transport_Semigroup_2020} used a recent adaptation of Harris' theorem in the sub-geometric, deterministic setting, due to Ca\~nizo and Mischler \cite{Canizo_Mischler_2020}, to obtain the optimal rate even in the case where the temperature varies, without symmetry hypothesis, with $M$ a wall Maxwellian. Regarding the case of the pure specular boundary condition, there is no mixing (the system is entirely deterministic), and we refer the interested reader to the thorough study of Briant \cite[Appendix A]{Briant_2015} focusing on the characteristics of the corresponding system. 
 
  To the best of our knowledge, this paper is the first analystic study of the asymptotics of the free-transport equation in a general domain with Cercignani-Lampis boundary condition. By adapting the method from \cite{Bernou_Transport_Semigroup_2020}, we obtain the optimal rate of convergence towards equilibrium of $\frac1{t^{d-}}$ in the $L^1$-norm. We hope that this understanding will help to tackle the difficult extension of the results regarding asymptotic behaviors of collisional kinetic equations to this more general boundary condition.

\subsection{Velocity flow}

In the pure diffuse case, that is when $(\rp, \rt) = (1,1)$, and for the Maxwell boundary condition, the steady flow of velocity  (perhaps surprisingly) vanishes, even in the case where the temperature is allowed to vary. This is not the case in general when one considers other parameters $(\rp, \rt) \ne (1,1)$. In particular, in the case of a gas confined between two plates with sinusoidal temperature distribution, while the steady flow vanishes for the Maxwell boundary condition, cf. \cite{Sone_Molecular_Gas_2007}, four different behaviors of this flow are observed when $\rp$ and $\rt$ vary. On this subject, the main reference is the work of Kosuge et al. \cite{Kosuge_steady_2011}. In Section \ref{Section:steady-flow}, we consider a model in which a gas is confined between two plates, one with pure diffuse reflection boundary condition $\rt = \rp = 1$, the second one with a general Cercignani-Lampis condition with variable temperature. We derive the steady state for the corresponding problem, giving the first example of an explicit steady state in the case $(\rp,\rt) \not \equiv (1,1)$, and we prove that this steady state implies no steady flow. A possible interpretation of this result is the following:  the pure diffuse boundary condition destroys the previous correlations, and the flow originated from it has no preferred orientation. This hints that the crucial mechanism behind the steady flow observed numerically by Kosuge et al. \cite{Kosuge_steady_2011} might be the absence of a decorrelation mechanism - in our toy model, the pure diffuse boundary condition, which plays a role for all trajectories.

\subsection{Hypotheses and main results} 

We assume that $\Omega \subset \RR^d$, with $d \in \{2,3\}$, and we endow $\RR^d$ with the Lebesgue measure. The symbols $\d x$, $\d v$ denote this measure. We assume that $\Omega$ is bounded and $C^2$ with closure $\bar{\Omega}$, and that the map $x \to n_x$ can be extended to the whole set $\bar{\Omega}$ as a $W^{1,\infty}(\Omega)$ map, where $W^{1,\infty}(\Omega)$ denotes the corresponding Sobolev space. For any $k \in \mathbb{N}^*$, we use the Euclidean norm on $\RR^k$ and denote $|x|$ the norm of $x$. We denote $x \cdot y$ the scalar product between $x$ and $y$ in $\RR^k$. We write $d(\Omega)$ for the diameter of $\Omega$, given by 
\[ d(\Omega) := \sup_{(x,y) \in \Omega^2} |x - y|. \]
On $\bar{G} = \bar{\Omega} \times \RR^d$, setting 
\[ \Sigma_0 := \big\{(x,v) \in \pO \times \RR^d, v \cdot n_x = 0 \big\}, \]
we define the map $\sigma$ by:
\begin{align}
\label{eq:def_sigma}
\sigma(x,v) = 
\left\{
\begin{array}{ll}
\inf \{t > 0, x + tv \in \pO\}, \qquad & (x,v) \in \Sigma_- \cup G, \\
0, & (x,v) \in \Sigma_+ \cup \Sigma_0,
\end{array}
\right.
\end{align}
which corresponds to the time of the first collision with the boundary for a particle in position $x$ with velocity $v$ at time $t = 0$. The $L^1$ space on $G$, denoted $L^1(G)$, is the space of measurable $\RR$-valued functions $f$ such that 
\[ \|f\|_{L^1} :=  \int_G |f(x,v)| \, \d v \d x < \infty. \]
For any non-negative measurable function $w$ defined on $G$, we introduce the weighted $L^1$ space $L^1_w(G) = \{f \in L^1(G), \|f w \|_{L^1} < \infty\}$ endowed with the norm defined by 
\[ \|f\|_w := \|f w \|_{L^1}. \]
 For any function $f \in L^1(G)$, we define the mean of $f$ by 
\begin{align}
\label{eq:def_mean}
\langle f \rangle = \int_G f(x,v) \, \d v \d x.
\end{align}
We assume that both accommodation coefficients are non-singular, i.e. $\rp \in (0,1)$ and $\rt \in (0,2)$. Note that this includes the case of the pure diffuse boundary condition. 
Finally, we assume that the wall temperature $\theta: \pO \to \RR_+^*$ is a continuous function, positive on $\pO$ compact, and thus admitting two extreme values $\theta_0, \theta_1 > 0$ such that 
\[ \forall x \in \pO, \qquad 0 < \theta_0 \le \theta(x) \le \theta_1. \]

The Harris' theorem used in this paper gives a convergence result in the $L^1$ norm depending on some weighted $L^1$ norm of the initial data. 
The weights will take the form of polynomials of the following quantity
\begin{align}
\label{eq:def_bracket}
\langle x,v \rangle := (1 + \sigma(x,v) + \sqrt{|v|}), \qquad (x,v) \in \bar{G}.
\end{align}
We set, for all $\alpha > 0$,
\[ m_{\alpha} := \langle x,v \rangle^{\alpha}. \]
After proving that the problem \eqref{eq:main_pb} is well-posed, we introduce the semigroup $(S_t)_{t \ge 0}$ such that, for all $f \in L^1(G)$, for all $t > 0$, $S_t f$ is the unique solution of \eqref{eq:main_pb} at time $t > 0$ belonging to $L^1(G)$. Our main result is the following:

\begin{thm}
\label{thm:Main}
For all $\epsilon \in (0,\frac12)$, there exists a constant $C > 0$ such that for all $t \ge 0$, for all $f,g \in L^1_{m_{d+1-\epsilon}}(G)$ with $\langle f \rangle = \langle g \rangle$, there holds
\[ \|S_t (f-g) \|_{L^1} \le \frac{C}{(1+t)^{d+1-\epsilon}} \|f - g\|_{m_{d+1-\epsilon}}. \]
\end{thm}

From this result, we deduce the existence of a unique steady state even in the case where the temperature varies. 

\begin{thm}
\label{thm:equilibrium}
There exists a unique $f_{\infty}$ such that, for all $\epsilon \in (0,\frac12)$, we have $f_{\infty} \in L^1_{m_{d - \epsilon}}(G)$, $0 \le f_{\infty}$, $\langle f_{\infty} \rangle = 1$, and 
\begin{align*}
v \cdot \nabla_x f_{\infty}(x,v) = 0, \qquad & (x,v) \in G, \\
\gamma_- f_{\infty}(x,v) = K \gamma_+ f_{\infty}(x,v), \qquad & (x,v) \in \Sigma_-.
\end{align*}
\end{thm}

Regarding the convergence towards the steady state, we can deduce the following corollary from an interpolation argument applied to the result of Theorem \ref{thm:Main}.

\begin{coroll}
\label{coroll:cvg-to-eq}
For all $\epsilon \in (0,\frac12)$, there exists a constant $C' > 0$ such that for all $t \ge 0$, for all $f \in L^1_{m_{d-\epsilon}}(G)$ with $\langle f \rangle = 1$, for $f_{\infty}$ given by Theorem \ref{thm:equilibrium}, 
\[ \|S_t(f-f_{\infty})\|_{L^1} \le \frac{C'}{(1+t)^{d-\epsilon}} \|f-f_{\infty}\|_{m_{d-\epsilon}}. \]
\end{coroll}

\begin{rmk}
As usual when using the subgeometric Harris' theorem, we can not apply directly Theorem \ref{thm:Main} to study the convergence towards the steady state, because we do not have in general $f_{\infty} \in L^1_{d+1-\epsilon}(G)$ for $\epsilon \in (0,\frac12)$. In particular, it is known that the explicit form in the case $\rt = \rp = 1$, $\theta \equiv 1$ is given by a Maxwellian which belongs to $L^1_{m_{d-\epsilon}}(G) \setminus L^1_{m_{d+1-\epsilon}}(G)$ for all $\epsilon \in (0,1)$. This limiting role of the steady state is well-known in the probabilistic counterpart of the theory used in this paper, see for instance Douc-Fort-Guillin \cite{Douc_Subgeometric_2009} and Hairer \cite{Hairer_2016}. 
\end{rmk}

\begin{rmk}
The hypothesis $f \in L^1_{m_{d-\epsilon}}(G)$ for some $\epsilon \in (0,\frac12)$ is satisfied if $f$ is bounded. For instance, the usual Maxwellian steady state of the pure diffure reflection satisfies this hypothesis.
\end{rmk}

\begin{rmk}
The conclusion from Corollary \ref{coroll:Interpol} is that the rate of convergence towards the steady state of the free-transport equation with Cercignani-Lampis boundary condition is better than $\frac{1}{t^{d}}$ (up to a log factor) when starting from an initial datum with enough regularity. As this is also the rate obtained for the pure diffuse boundary condition (see for instance \cite{kuo_free_2013} for the spherically symmetric case, and \cite{Bernou_Transport_Semigroup_2020}, \cite{bernou_fournier_collisionless_2019} for the general case), which corresponds to the particular case $\rp = \rt = 1$, and since it is known that this rate is optimal in this context, we can conclude to the optimality for the general Cercignani-Lampis boundary condition. 
\end{rmk}

\begin{rmk}
Our proof of Theorem \ref{thm:Main} (and thus of Corollary \ref{coroll:Interpol}) is constructive, i.e. the constant $C$ appearing in Theorem \ref{thm:Main} can be computed explicitely, although it might depend in a very complicated manner of the geometry of $\Omega$. An interesting fact is that the proof requires some control of the flux of the solution at the boundary, provided by Lemma \ref{lemma:control_flux}. The constant appearing in this flux takes the form $\frac{M}{1-m}$, with $m$ a positive power of $\max((1-\rp),(1-\rt)^2)$, and $M$ a constant independent of $\rp$ and $\rt$. Unsurprisingly as $(\rp, \rt) \to (0,0)$ (i.e., as we retain more and more information from the incoming velocities, converging towards the pure specular boundary condition), this constant grows and at the limit we lose the control of the flux. The same occurs as $(\rp,\rt) \to (0,2)$, i.e. as we converge towards the bounce-back boundary condition. 
\end{rmk}

\begin{rmk}
Rather than weights in the form of power of  
\[ \langle x, v \rangle = (1 + \sigma(x,v) + \sqrt{|v|}) \]
we can extend all three results to weights in the form of power of 
\[ _\delta \langle x, v \rangle = (1 + \sigma(x,v) + |v|^{2 \delta}) \]
for any $\delta \in (0,\frac12)$. The rates of convergence are then unmodified, although the constants appearing in front of them change.
\end{rmk}

\subsection{A toy model for the study of the velocity flow}
\label{subsec:toy_model}

In Section \ref{Section:steady-flow}, we study the free-transport equation in the box $[0,1]^2 \subset \RR^2$ with periodic boundary conditions at $x_1 = 0$ and $x_1 = 1$ and two Cercignani-Lampis boundary conditions at $x_2 = 0$ and $x_2 = 1$. Hence the model is close, in spirit, to the one presented by Kosuge et al \cite{Kosuge_steady_2011} on their work on the velocity flow. We allow $\rt$ and $\rp$ to vary with the boundary, taking $\rt = \rp = 1$ at $x_2 = 1$ and $\rp = \rt(2-\rt)$ with $\rp \in (0,1)$ at $x_2 = 0$. Therefore we have a pure diffuse reflection at $x_2 = 1$ and a more general Cercignani-Lampis boundary condition at $x_2 = 0$. We set the temperature to be $1$ at $x_2 = 1$ and we take $\theta_2 : (x_1,0) \to (1,\infty)$ to be the function giving the temperature at $x_2 = 0$. 

With this at hand, we provide an explicit steady state for this problem, giving a first instance of an explicit steady state for a problem in which the Cercignani-Lampis boundary condition with $(\rt,\rp) \not = (1,1)$ is considered. We also prove that this steady state exhibits no velocity flow, hinting that the presence of a piece of the boundary in which a decorrelation mechanism takes place (the pure diffuse boundary condition) might suffice to cancel all such flows. We plan to pursue in the near future, with probabilistic methods, the rigorous investigation of the velocity flow for models involving a Cercignani-Lampis boundary condition.

\subsection{Proof strategy}

The key result of this paper is Theorem \ref{thm:Main}. Its proof is purely deterministic: although we use some known facts from probability theory to shorten some computations, those could be adapted to be written entirely without this framework. We adapt the method of \cite{Bernou_Transport_Semigroup_2020}, more precisely we prove a subgeometric Harris' theorem for the particular choice of weights involved here. The idea of this deterministic adaptation to the previously known probabilistic results of Douc-Fort-Guillin \cite{Douc_Subgeometric_2009} and Hairer \cite{Hairer_2016} is due to Ca\~nizo and Mischler \cite{Canizo_Mischler_2020}. We provide a self-contained proof, except for the interpolation arguments which are taken directly from \cite{Bernou_Transport_Semigroup_2020}. Let us detail the approach, and the main adaptations required to handle the more involved Cercignani-Lampis boundary condition compared to the Maxwell boundary condition treated in \cite{Bernou_Transport_Semigroup_2020}. 

We introduce the operator $\mathcal{L}$ such that \eqref{eq:main_pb} rewrites as a Cauchy problem:
\begin{align*}
\left\{
\begin{array}{lll}
\partial_t f &= \mathcal{L} f & \hbox{ in } \RR_+ \times \bar{\Omega} \times \RR^d, \\
f(0,\cdot) &= f_0(\cdot) & \hbox{ in } G.
\end{array}
\right.
\end{align*}
There are two main tools to prove a subgeometric Harris' theorem for such a problem. The first one is to derive an inequality of the form
\begin{align*}
\mathcal{L}^* w_1 \le - w_0 + \kappa,
\end{align*} 
for some $\kappa > 0$, for $\mathcal{L}^*$ the adjoint operator of $\mathcal{L}$, for some weights $(w_0,w_1)$ with $1 \le w_0 \le w_1$. Typically one wants to obtain several inequalities of this kind, with various choices of weights instead of $(w_0,w_1)$. In our case, such inequality is very hard, perhaps impossible, to derive. On the other hand, we can obtain an integrated version of the inequality, i.e. the existence of two constants $b_1,C_1 > 0$ such that for all $T > 0$, $f \in L^1_{m_{d+1-\epsilon}}(G)$,
\begin{align}
\label{eq:ineq_Lyap_Intro}
\| S_T f\|_{m_{d+1-\e}} + C_1 \int_0^T \|S_s f\|_{m_{d-\e}} ds \le \|f\|_{m_{d+1-\e}} + b_1(1+T) \|f\|_{L^1}.
\end{align}
The existence of such weights relies heavily on the fact that
\[ v \cdot \nabla_x \sigma(x,v) = -1, \]
as noticed for instance by Esposito, Guo, Kim and Marra \cite{Esposito2013}. This approach was also taken in \cite{Bernou_Transport_Semigroup_2020}, however, there is, in the case of the Cercignani-Lampis boundary condition, a key difficulty in the control of the flux compared to the case of the diffuse boundary condition. While, in the latter, we had the inequality
\begin{align*}
\int_0^T \int_{\pO} \int_{\Sigma^x_+ } |v \cdot n_x| \gamma_+ |S_t f|(x,v) \, \d v \d \zeta(x) \d s \leq C (1+T)\|f\|_{L^1}, 
\end{align*}
for some $C > 0$, where $d\zeta(x)$ is the surface measure at $x \in \pO$, such an inequality does not hold in our context. Instead, we derive a partial control of the flux in Lemma \ref{lemma:control_flux}, given, for all $\Lambda > 0$, by the existence of a constant $C_{\Lambda} > 0$ such that 
\begin{align*}
\int_0^T \int_{\pO} \int_{\{v \in \Sigma^x_+, |v| \leq \Lambda \}} |v \cdot n_x| \gamma_+ |S_t f|(x,v) \, \d v \d \zeta(x) \d s \leq C_\Lambda (1+T)\|f\|_{L^1},
\end{align*}
and on the fact that, since $\rp \in (0,1)$ and $(1-\rt)^2 \in (0,1)$, the outcoming velocity has, on average, a smaller norm than the incoming one. 

The second ingredient to adapt the subgeometric Harris' theory to our context is a positivity result, the Doeblin-Harris condition, for the semigroup $(S_t)_{t \ge 0}$. This is given by Theorem \ref{thm:Doeblin-Harris} in the form of the following inequality: for any $\Lambda \ge 2$, there exist $T(\Lambda) > 0$ and a non-negative, non-trivial measure $\nu$ on $G$ with $\nu \not \equiv 0$ such that for all $(x,v) \in G$, for all $f_0 \in L^1(G)$, $f_0 \ge 0$, 
\begin{align}
\label{eq:DH-Intro}
S_{T(\Lambda)} f_0(x,v) \ge \nu(x,v) \int_{\{(y,w) \in G, \langle y, w \rangle \le \Lambda\}} f_0(y,w) \, \d y \d w.
\end{align} 
To prove Theorem \ref{thm:Main}, we combine the two results \eqref{eq:ineq_Lyap_Intro} and \eqref{eq:DH-Intro} as in \cite{Canizo_Mischler_2020, Bernou_Transport_Semigroup_2020}. We assume that $g=0$ so that $f \in L^1_{m_{d+1-\e}}(G)$ with $\langle f \rangle = 0$,  and for $T > 0$ large enough we introduce the modified norm
\[ \vertiii{.}_{m_{d+1-\e}} = \|.\|_{L^1} + \beta \|.\|_{m_{d+1-\e}} + \alpha \|.\|_{m_{d-\e}} \]
for two constants $\alpha, \beta > 0$ well-chosen, depending on $T$. We prove first a contraction result for this new norm
\begin{align}
\label{eq:ineq_modified_norm_intro}
\vertiii{S_T f}_{m_{d+1-\e}} \le \vertiii{f}_{m_{d+1-\e}}.
\end{align}
Then, we introduce two auxiliary weights so that $1 \le w_0 \le w_1 \le m_{d+1-\e}$ for which, with a similar argument, for some modified norm $\vertiii{.}_{w_1}$, for $T > 0$ as above and for $\tilde{\alpha} > 0$ constant, we can derive the following inequality
\begin{align}
\label{eq:ineq_modified_w1_intro}
\vertiii{S_T f}_{w_1} + 2 \tilde \alpha \|f\|_{w_0} \le \vertiii{f}_{w_1}. 
\end{align}
We combine \eqref{eq:ineq_modified_norm_intro} and \eqref{eq:ineq_modified_w1_intro} repeatedly and use the inequalities between the weights to conclude. 

Once Theorem \ref{thm:Main} is established, the proof of Theorem \ref{thm:equilibrium} follows from a refined version of \eqref{eq:ineq_modified_norm_intro}, and Corollary \ref{coroll:cvg-to-eq} is derived from Theorem \ref{thm:equilibrium} via an interpolation argument. 

The proof of the results mentioned in Subsection \ref{subsec:toy_model} are obtained directly by studying the candidate steady state which is itself obtained by the method of characteristics. While the computations are easy in the case where the temperature is constant, a few tricks are necessary when it is allowed to vary. They rely heavily on earlier computations performed by Chen \cite{Chen_CL_2020}. 

\subsection{Plan of the paper}

In Section \ref{section:setting}, we show that the problem \eqref{eq:main_pb} is well-posed, that the associated semigroup is a contraction in $L^1(G)$, we prove that the Cercignani-Lampis boundary condition is not regular in the sense of \cite{Lods_2020} and we introduce some probabilistic tools. With the help of those, we prove in Section \ref{section:Lyapunov} the inequality \eqref{eq:ineq_Lyap_Intro} for a variety of weights of the form $m_{\gamma}$, $\gamma \in (1,d+1)$, deriving along the way the partial control of the flux mentioned above. The inequality \eqref{eq:DH-Intro} is derived in Section \ref{section:DH}. The proofs of Theorems \ref{thm:Main}, \ref{thm:equilibrium} and Corollary \ref{coroll:cvg-to-eq} are given in Section \ref{Section:steady-flow}, starting from the one of Theorem \ref{thm:Main}, from which Theorem \ref{thm:equilibrium} and then Corollary \ref{coroll:cvg-to-eq} are obtained. Finally, Section \ref{Section:steady-flow} is devoted to the study of our toy model.

\section{Setting, elementary properties, preliminary notions}
\label{section:setting}

\subsection{Notations and associated semigroup}

We first set some notations. We write $\bar{B}$ for the closure of any set $B$. We denote by $\mathcal{D}(E) := C^1_c(E)$ the space of test functions, $C^1$ with compact support, on $E$. We write $\d \zeta(x)$ for the surface measure at $x \in \pO$. We denote by $\mathcal{H}$ the $d-1$ dimensional Hausdorff measure. 

For a function $f \in L^{\infty}([0,\infty); L^1(\Omega \times \RR^d))$, admitting a trace $\gamma f$ at the boundary, we write $\gamma_{\pm} f$ for its restriction to $(0,\infty) \times \Sigma_{\pm}$. This corresponds to the trace obtained in Green's formula, see Mischler \cite{Mischler_Vlasov_1999}.
Note first that the boundary operator $K$ given by \eqref{eq:def_K} has norm $1$. This follows easily from the normalization property \eqref{eq:normalization_basic}:

\begin{lemma}[$K$ is non-negative and stochastic]
\label{lemma:K}
The boundary operator $K$ defined by \eqref{eq:def_K} is non-negative, and satisfies, for all $t \ge 0$, $x \in \pO$, for all $f$ regular enough so that both integrals are well-defined,
\begin{align}
\label{eq:normK1}
\int_{\Sigma_-^x} K\gamma_+ f(t,x,v) \, |v \cdot n_x| \, \d v = \int_{\Sigma_+^x} \gamma_+ f(t,x,v) \, |v \cdot n_x| \d v.
\end{align}
\end{lemma}

\begin{proof}
The non-negativity of $K$ is straightforward in view of \eqref{eq:def_K} and \eqref{eq:def_R}. Recall from \eqref{eq:normalization_basic} that, for all $x \in \pO$, $u \in \Sigma_+^x$, 
\begin{align}
\label{eq:normalization}
 \int_{\Sigma_-^x} R(u \to v;x) \, |v \cdot n_x| \, \d v = 1.
 \end{align}
 Hence, 
 \begin{align*}
 \int_{\Sigma_-^x} K\gamma_+ f(t,x,v) \, |v \cdot n_x| \, \d v &= \int_{\Sigma_-^x} |v \cdot n_x| \Big( \int_{\Sigma_+^x} \gamma_+ f(t,x,u) \, |u \cdot n_x| \, R(u \to v;x) \, \d u \Big) \d v \\
 &= 
 \int_{\Sigma_+^x} |u \cdot n_x| \, \gamma_+ f(t,x,u) \Big( \int_{\Sigma_-^x} R(u \to v; x) \, |v \cdot n_x| \, \d v \Big) \d u
 \end{align*}
 where we used Fubini's theorem, and the conclusion follows. 
\end{proof}

Since the boundary operator is conservative and stochastic, the problem \eqref{eq:main_pb} is governed by a $C_0$-stochastic semigroup $(S_t)_{t \ge 0}$, i.e. a non-negative, mass-conservative semigroup such that, for $f_0 \in L^1(G)$, for all $t \ge 0$, $S_t f_0 = f(t,\cdot)$ is the unique solution in $L^{\infty}([0,\infty); L^1(G))$ to \eqref{eq:main_pb} taken at time $t$. For the sake of completeness, we check those two properties and show that $(S_t)_{t \ge 0}$ is a contraction semigroup in the following theorem.

\begin{thm}[Positivity and mass conservation \cite{Cercignani_Lampis_1971}]
\label{thm:mass_&_L1_contraction}
Let $f \in L^1(G)$. For all $t \ge 0$, $\langle S_t f \rangle = \langle f \rangle$. Moreover, we have
\[ \|S_t f\|_{L^1} \le \|f\|_{L^1},\]
 and, if $f$ is non-negative, so is $S_t f$. 
\end{thm}

\begin{proof}

\textbf{Step 1.} We write $f(t,x,v)$ for $S_t f(x,v)$ for all $(t,x,v) \in [0,\infty) \times G$, $\gamma f$ for the corresponding trace on $(0,\infty) \times \Sigma$. Using Green's formula, we have, for all $t \ge 0$,
\[ \frac{d}{dt} \int_G f(t,x,v) \, \d v \d x = - \int_G v \cdot \nabla_x f(t,x,v) \, \d v \d x = - \int_{\Sigma} \gamma f(t,x,v) \, (v \cdot n_x) \, \d v \d \zeta(x), \]
and, using \eqref{eq:normK1} as well as the boundary condition satisfied by $f$, we conclude that 
\begin{align*}
\frac{d}{dt} \langle S_t f \rangle = 0.
\end{align*}

\vspace{.3cm}

\textbf{Step 2.} By triangle inequality, for almost all $t \ge 0$, $x \in \pO$,
\begin{align*}
\int_{\Sigma_-^x} |K f(t,x,v)| \, |v \cdot n_x| \, \d v &= \int_{\Sigma_-^x} |v \cdot n_x| \Big| \int_{\Sigma_+^x} f(t,x,u) \, |u \cdot n_x| \, R(u \to v; x) \, \d u \Big| \, \d v \\
&\le \int_{\Sigma_+^x} |f(t,x,u)| \, |u \cdot n_x| \Big( \int_{\Sigma_+^x} |v \cdot n_x| \, R(u \to v; x) \, \d v \Big) \d u \\
&= \int_{\Sigma_+^x} |f(t,x,u)| \, |u \cdot n_x| \, \d u,
\end{align*}
where we used the positivity of $R$ and Tonelli's theorem to derive the inequality. Using Green's formula and the equation satisfied by $f$, we find
\begin{align*}
\frac{d}{dt} \int_G |f(t,x,v)| \, \d v \d x \le \int_{\Sigma} |f(t,x,u)| \, (u \cdot n_x) \, \d u \d \zeta(x) 
\end{align*}
and combining the boundary condition satisfied by $f$ with the previous inequality, we conclude that
\begin{align*}
\frac{d}{dt} \|S_t f\|_{L^1} \le 0. 
\end{align*}

\vspace{.3cm}

\textbf{Step 3: Positivity.} Note that $(S_t f)_- = \frac{|S_t f| - S_t f}{2}$. Assume that $f \ge 0$, then $f_- = 0$, and, for all $t \ge 0$, since $\langle S_t f \rangle = \langle f \rangle$ (by Step 1) and since $(S_t)_{t \ge 0}$ is a contraction in $L^1$ (by Step 2),
\begin{align*}
\| (S_t f)_-\|_{L^1} &= \int_G \frac{|S_t f| - S_t f}{2} \, \d v \d x  \\
&= 
\frac12 \Big( \|S_t f\|_{L^1} - \langle S_t f \rangle \Big) \\
&\le 
\frac12 \Big( \|f\|_{L^1} - \langle f \rangle \Big) = \int_G \frac{|f| - f}{2} \, \d v \d x = \|f_-\|_{L^1} = 0,
\end{align*}
and since $(S_t f)_- \ge 0$ almost everywhere (a.e.) on $G$, we conclude that $(S_t f)_- = 0$ a.e. on $G$. 
\end{proof}

In the remaining part of this paper, we will investigate the decay properties of the problem at the level of this semigroup $(S_t)_{t \ge 0}$.

\subsection{Probabilistic facts and regularity}

We briefly present the Rice distribution and a connection to Gaussian random variables. For a deeper exposition of this probabilistic material, we refer to Kobayashi, Mark and Turin \cite[Section 7.5.1 and 7.5.2]{Kobayashi_2009}. We write $Y \sim \mathcal{N}(m, \Delta)$ when $Y$ is a Gaussian random vector on $\RR^n$, $n \ge 1$ with mean $m \in \RR^n$ and co-variance matrix $\Delta \in \mathcal{M}_s^n$ the space of symmetric matrices of size $n \times n$, and we write $I_n$ for the identity matrix of size $n \times n$. If $X$ and $Y$ are two random variables, we write $X \overset{\mathcal{L}}{=} Y$ if $X$ and $Y$ have the same distribution. 

\begin{defi}
\label{defi:Rice} 
Let $\mu \in \RR$, $\sigma^2 > 0$. We say that $X$ follows a Rice distribution of parameter $(\mu, \sigma^2)$ and write $X \sim \mathrm{Ri}(\mu,\sigma^2)$ if $X$ has the following density with respect to the Lebesgue measure:
\[ f_{\mathrm{Ri}(\mu,\sigma^2)}(x) = \frac{x}{\sigma^2} e^{-\frac{x^2}{2 \sigma^2}} e^{-\frac{\mu^2}{2 \sigma^2}} I_0 \Big(\frac{\mu x}{\sigma^2} \Big), \qquad x \in \RR_+. \]
\end{defi}

\begin{prop}[\cite{Kobayashi_2009}]
\label{prop:Rice}
Let $\mu \in \RR$, $\sigma^2 > 0$ and $\vartheta \in [0, 2\pi)$. Let $X_1 \sim \mathcal{N}(\mu \cos(\vartheta),\sigma^2)$, $X_2 \sim \mathcal{N}(\mu \sin(\vartheta),\sigma^2)$ be two independent random variables. Let $Y \sim \mathrm{Ri}(\mu,\sigma^2)$. Then
\[ \sqrt{X_1^2 + X_2^2} \overset{\mathcal{L}}{=} Y. \]
\end{prop}

Let us conclude this section by a proof that the Cercignani-Lampis boundary condition does not fall into the framework of \cite{Lods_2020}:

\begin{prop}
\label{prop:not_regular}
We have, for all $x \in \pO$,
\begin{align*}
\lim \limits_{m \to \infty} \sup_{v' \in \Sigma_+^x} \int_{\{v \in \Sigma_-^x, |v| \ge m\}} R(v' \to v; x) |v \cdot n_x| \d v \ge \Big(\frac12 \Big)^d >  0.
\end{align*}
In particular, \cite[Equation (3.4)]{Lods_2020} is not satisfied, and the boundary operator $K$ is not a regular diffuse operator in the sense of \cite{Lods_2020}. 
\end{prop}

\begin{proof}
We note first that
\[ \{ v \in \Sigma_-^x, |v| \ge m\} \supset \Big\{ v \in \Sigma_-^x, |\vp| \ge \frac{\sqrt{2}m}{2}, |\vt|  \ge \frac{\sqrt{2}m}{2} \Big\}, \]
so that we have, for all $m > 0$, $u \in \Sigma_+^x$, 
\begin{align*}
\int_{\{v \in \Sigma_-^x, |v| \ge m\}} R(u \to v; x) |v \cdot n_x| \d v \ge \int_{\{ v \in \Sigma_-^x, |\vp| \ge \frac{\sqrt{2}m}{2}, |\vt|  \ge \frac{\sqrt{2}m}{2} \}} R(u \to v; x) |v \cdot n_x| \d v. 
\end{align*}
We note that, with the previous definitions
\[ |\vp| R(u \to v; x) = f_{\mathrm{Ri}((1-\rp)^{\frac12} |\up|, \theta(x) \rp)}(-\vp) f_{\mathcal{N}((1-\rt) \ut, \theta(x) \rt (2 - \rt) I_{d-1})}(\vt). \]
We assume from now on, without loss of generality, that $n_x = e_1$. We can thus write, with the change of variable sending $v_1$ to $-v_1$ and splitting the integral
\begin{align}
\label{eq:intermed_prelim_result}
\int_{\{v \in \Sigma_-^x, |v| \ge m\}}& R(u \to v; x) |v \cdot n_x| \d v \\
&\ge \Big( \int_{\frac{\sqrt{2}m}{2}}^{\infty} f_{\mathrm{Ri}((1-\rp)^{\frac12} |\up|, \theta(x) \rp)}(v_1) \d v_1 \Big) \nonumber \\
&\qquad \times  \Big( \int_{\{\vt \in \RR^{d-1}, |\vt| \ge \frac{\sqrt{2}m}{2}\}} f_{\mathcal{N}((1-\rt) \ut, \theta(x) \rt (2 - \rt) I_{d-1})}(\vt) \d \vt \Big). \nonumber
\end{align} 
where we abusively identified $\vp$ with $v_1$ and $\vt$ with $(v_2,\dots,v_d)$ since $n_x = e_1$. 
Choosing $\ut = (\frac{\sqrt{2}m}{2 (1-\rt)}, \dots,\frac{\sqrt{2}m}{2 (1-\rt)})$ in $\RR^{d-1}$, we clearly have  
\[ \int_{\{\vt \in \RR^{d-1}, |\vt| \ge \frac{\sqrt{2}m}{2}\}} f_{\mathcal{N}((1-\rt) \ut, \theta(x) \rt (2 - \rt) I_{d-1})}(\vt) \d \vt \ge \Big(\frac12 \Big)^{d-1}, \]
by properties of Gaussian random variables: this follows by splitting the integral into $d-1$ integrals over $\RR$ of the form
\[ \int_{\{|v|  \ge \frac{\sqrt{2}m}{2} \}} \frac{e^{-\frac{\big(v - \tfrac{\sqrt{2}m}{2}\big)^2}{2 \theta(x) \rt (2-\rt)}}}{\sqrt{2 \pi \theta(x) \rt (2 - \rt) }}\d v \ge \int_0^{\infty}  \frac{e^{-\frac{v^2}{2 \theta(x) \rt (2-\rt)}}}{\sqrt{2 \pi \theta(x) \rt (2 - \rt) }}\d v = \frac12,  \]
where we only kept the integral over a subset of $\RR_+$ and performed the change of variable $v' = v - \tfrac{\sqrt{2}m}{2}$. 
  As for the first integral on the right-hand side of \eqref{eq:intermed_prelim_result}, we have
\begin{align*}
\int_{\frac{\sqrt{2}m}{2}}^{\infty} &f_{\mathrm{Ri}((1-\rp)^{\frac12} |\up|, \theta(x) \rp)}(v_1) \d v_1
\\ &= 
1 - \int_0^{\frac{\sqrt{2}m}{2}} |v_1| \frac{e^{-\frac{|v_1|^2}{2 \theta(x) \rp}}}{\theta(x) \rp} e^{-\frac{(1-\rp)|\up|^2}{2 \theta(x) \rp}} \frac1{\pi} \int_0^{\pi} e^{- \cos(\phi) \frac{(1-\rp)^{\frac12} |\up| v_1}{\theta(x) \rp}} \d \phi \d v_1 \\
&= 1 - \frac1{\pi} \int_0^{\frac{\sqrt{2}m}{2}} |v_1| \frac{e^{-\frac{|v_1|^2}{2 \theta(x) \rp}}}{\theta(x) \rp} \Big( \int_0^{\pi} e^{-\frac{|\up|^2 \big((1-\rp) + 2 \cos(\phi) (1-\rp)^{\frac12} \frac{|\up|}{|\up|^2} v_1 \big) }{2 \theta(x) \rp}}  \d \phi \Big) \d v_1,
\end{align*}
and an application of the dominated convergence theorem clearly shows that the last term on the right-hand-side converges to $0$ as $|\up| \to \infty$. Hence there exists $\up$ with $|\up|$ large enough so that 
\[ \int_{\frac{\sqrt{2}m}{2}}^{\infty} f_{\mathrm{Ri}((1-\rp)^{\frac12} |\up|, \theta(x) \rp)}(v_1) \d v_1 \ge \frac12. \]
Since we can find such a couple $(\up,\ut)$ for all $m > 0$, the conclusion follows. 
\end{proof}

%
%
%

\section{Subgeometric Lyapunov condition}
\label{section:Lyapunov}
Recall the definition of the map $\sigma$ from \eqref{eq:def_sigma}. On $\bar{G}$, we define the function $q$ by
\begin{align}
\label{eq:def_q}
q(x,v) = x + \sigma(x,v) v.
\end{align}

In terms of characteristics of the free-transport equation, for $(x,v) \in \bar{G}$, $q(x,v)$ corresponds to the right limit in $\bar{\Omega}$ of the characteristic with origin $x$ directed by $v$. The real number $\sigma(x,v)$ corresponds to the time at which this characteristic reaches the boundary, if it started from $x$ at time $0$ with velocity $v$ with $x \in \Omega$ or $x \in \pO, v \cdot n_x < 0$. If $x \in \pO$ and $v$ is not pointing towards the gas region (that is, $(x,v)$ is already the right limit of the corresponding characteristic), $q(x,v)$ simply denotes $x$. 

We recall from Esposito, Guo, Kim and Marra \cite[Lemma 2.3]{esposito_non-isothermal_2013}, that
\[v \cdot \nabla_x \sigma(x,v) = - 1, \]
for all $(x,v) \in G$.
This minus sign can be understood in the following way: since $\sigma(x,v)$ is the time needed for a particle in position $x \in \bar{\Omega}$ with velocity $v \in \RR^d$ to hit the boundary starting from the time $t = 0$, moving the particle from $x$ along the direction $v$ reduces this time. 

Recall the definition of the bracket $\langle x,v \rangle$ for $(x,v) \in \bar G$ from \eqref{eq:def_bracket} and that for all $k > 0$, $m_k(x,v) = \langle x,v \rangle^k$. This section is devoted to the proof of the following proposition.

\begin{prop}
\label{prop:ineq_lyapunov}
For any $\alpha \in (1,d+1)$, there exists $b > 0$ explicit, depending on $\alpha$, such that for all $T > 0$, $f \in L^1_{m_\alpha}(G)$, 
\begin{align}
\label{eq:main_ineq_lyapunov}
\| S_T f\|_{m_\alpha} + \alpha \int_0^T \|S_s f\|_{m_{\alpha - 1}} \d s \le \|f\|_{m_\alpha} + b(1 + T) \|f\|_{L^1}.
\end{align}
\end{prop}

To derive this result, we first need to obtain some control of the flux. This is the main source of additional difficulty compared to the pure diffuse case of \cite{Bernou_Transport_Semigroup_2020}. We tackle this issue in Lemmas \ref{lemma:control_flux} and \ref{lemma:Iux}.

\begin{lemma}[Control of the flux]
\label{lemma:control_flux}
For all $\Lambda > 0$, there exists an explicit constant $C_{\Lambda} > 0$ such that for all $f \in L^1(G)$, $T > 0$,
\begin{align*}
\int_0^T \int_{\pO} \int_{\{v \cdot n_x > 0, |v| \le \Lambda\}} |\vp| \,\gamma_+ |S_s f|(x,v) \, \d v \d \zeta(x) \d s \le C_{\Lambda} (1 + T) \|f\|_{L^1}. 
\end{align*}
\end{lemma}

\begin{proof}
We have, by definition of $(S_t)_{t \ge 0}$, that
\[ \partial_t |S_t f| + v \cdot \nabla_x |S_t f| = 0, \qquad \text{a.e. in } [0,T] \times G. \]
Recall that $x \to n_x$ is a $W ^{1,\infty}(\Omega)$ map by hypothesis. Multiplying this equation by $(v \cdot n_x)$ and integrating on $[0,T] \times \Omega \times \{v \in \RR^d, |v| \le 1\}$, we find
\begin{align*}
0 = \int_0^T \int_{\Omega} \int_{\{|v| \le 1\}} (v \cdot n_x) \, \Big( \partial_t + v \cdot \nabla_x \Big) |f|(t,x,v) \, \d v \d x \d t.
\end{align*}
Integrating by parts in both time and space on the right-hand side, we find
\begin{align*}
0 = &\Big[ \int_{\Omega} \int_{\{|v| \le 1\}} (v \cdot n_x) \, |f|(t,x,v)  \, \d v \d x \Big]_0^T \\
&- \int_0^T \int_\Omega \int_{\{|v| \le 1\}} |f| (t,x,v) \, v \cdot \nabla_x (v \cdot n_x) \, \d v \d x \d t  \\
&+ \int_0^T \int_{\pO} \int_{\{|v| \le 1\}} |v \cdot n_x|^2 \, \gamma |f|(t,x,v) \,  \d v \d \zeta(x) \d t, 
\end{align*}
where we used that 
\begin{align}
\label{eq:absolute_trace}
|\gamma S_t f(x,v)| = \gamma |S_t f|(x,v) \qquad \text{ a.e. in } ((0,\infty) \times \Sigma_+) \cup ((0,\infty) \times \Sigma_-),
\end{align}
see Mischler \cite[Corollary 1]{Mischler_Vlasov_1999}.
Since $x \to n_x$ belongs to $W^{1,\infty}(\Omega)$ and using the triangle inequality, this leads to 
\begin{align*}
\int_0^T \int_{\{ (x,v) \in \Sigma_-, |v| \le 1\}} \hspace{-.6cm} |\vp|^2 \,  \gamma_- |f|(t,x,v) \d v \d \zeta(x) \d t \leq 2 \|f\|_{L^1} + \|n_{\cdot}\|_{W^{1,\infty}} \int_0^T \|S_s f\|_{L^1} \, \d s. 
\end{align*}
Using the boundary condition and that  for all $s \ge 0$, $\|S_s f\|_{L^1} \le \|f\|_{L^1}$, we find 
\begin{align}
\label{eq:before_J}
\int_0^T  \int_{\{(x,u) \in \Sigma_+, |u| \le \Lambda\}} |\up| \, \gamma_+ |f|(t,x,u) \int_{\{v \in \Sigma_-^x, |v| \le 1\}} |\vp|^2 \,& R(u\to v; x) \, \d v \d u \d \zeta(x) \d t \nonumber \\
&\leq C (1 + T) \|f\|_{L^1},
\end{align}
for some $C > 0$ independent of $T$, where we used that $\{ (x,u) \in \Sigma_+, |u| \le \Lambda \} \subset \Sigma_+$. 
We claim that there exists $c_{\Lambda} > 0$ such that for all $(x,u) \in \Sigma_+$ with $|u| \le \Lambda$, 
\[ J_{u,x} := \int_{\{v \in \Sigma_-^x, |v| \le 1\}} |\vp|^2 \, R(u\to v;x) \d v \ge c_{\Lambda}. \]
Indeed, 
\begin{align*}
J_{u,x} =  \int_{\{v \in \Sigma_-^x, |v| \le 1\}} &\frac{|\vp|^2}{\theta(x) \rp (2 \pi \theta(x) \rt (2 - \rt))^{\frac{d-1}{2}}}  e^{-\frac{|\vp|^2}{2 \theta(x) \rp}} e^{-\frac{(1-\rp)|\up|^2}{2 \theta(x) \rp}} \\\
& \times I_0 \Big(\frac{ (1-\rp)^{\frac12} \up \cdot \vp}{ \theta(x) \rp} \Big) e^{-\frac{|\vt - (1-\rt) \ut|^2}{2 \theta(x) \rt (2 - \rt)}} \d v,
\end{align*} 
and, since $x \to n_x$ and $x \to \theta(x)$ are continuous, $(x,u) \to J_{u,x}$ is clearly continuous with $J_{u,x} > 0$ on the compact set $\{(x,u) \in \Sigma_+, |u| \le \Lambda\}$.
Therefore, there exists $c_{\Lambda} > 0$ such that for all $(x,u) \in \{(x,u) \in \Sigma_+, |u| \le \Lambda\}$, 
\begin{align*}
J_{u,x} \ge c_{\Lambda}.
\end{align*}
Note that, for any given $\Lambda$, the value of $c_{\Lambda}$ can be computed explicitly.
Inserting this into \eqref{eq:before_J}, we find 
\begin{align*}
c_{\Lambda} \int_0^T \int_{\{(x,v) \in \Sigma_+, |v| \le \Lambda\}} |\vp| \, \gamma_+ |f|(t,x,v) \, \d v  \d \zeta(x) \d t \le C (1 + T) \|f\|_{L^1},
\end{align*}
and the conclusion follows by setting $C_{\Lambda} = \frac{C}{c_{\lambda}} > 0$. 
\end{proof}

\begin{rmk}
The fact that we only obtained a partial control on the flux, instead of a control of the whole quantity
\begin{align*}
\int_0^T \int_{\Sigma_+} |v \cdot n_x| \gamma_+ |S_s f|(x,v) \, \d v \d \zeta(x) \d s
\end{align*}
is closely related to the lack of weak compactness of the operator $K$ obtained in Proposition \ref{prop:not_regular}. 
\end{rmk}

\begin{lemma}
\label{lemma:Iux}
Let $\alpha \in (1,d+1)$. For all $C > 0$, there exists $\Lambda > 0$ such that for all $x \in \pO$, $u \in \Sigma^x_+$ with $|u| \ge \Lambda$,
\begin{align}
\label{eq:conclusion_lemma_flux_ineq} 
I_{u,x} := \int_{ \Sigma_-^x} |\vp| \, \Big\{ (1 + d(\Omega) + \sqrt{|v|})^{\alpha} - (1 + \sqrt{|u|})^{\alpha} \Big\} \, R(u \to v; x) \, \d v \le - C. 
\end{align}
\end{lemma}

\begin{proof}
Although the result can be derived by purely deterministic arguments with the same idea, we will use insights from probability theory for the sake of conciseness. Recall that if $k \ge 2$, $\mu \in \RR^k$, $\Sigma \in \mathcal{M}_s^k$ and $N \sim \mathcal{N}(\mu,\Sigma)$, $\bar{N} := N - \mu \sim \mathcal{N}(0, \Sigma)$. Second, we recall that we write $X \sim \mathrm{Ri}(\mu, \sigma^2)$ if $X$ follows the Rice distribution of parameters $\mu \in \RR$, $\sigma^2 > 0$, see Definition \ref{defi:Rice}, and that we denote $f_{\mathrm{Ri}(\mu,\sigma^2)}$ the corresponding density on $\RR_+$. Finally, we recall the result from Proposition \ref{prop:Rice} which links Gaussian random variables and Rice distributions.

Note that
\begin{align*}
 I_{u,x} 
 &= \int_{\Sigma_-^x} |\vp| \, \Big\{\big(1 + d(\Omega) + (|\vp|^2 + |\vt|^2)^{\frac14} \big)^{\alpha} - (1 + \sqrt{|u|})^{\alpha} \Big\} \, R(u \to v; x) \, \d v,
\end{align*}
by definition of $\vp$ and $\vt$. Since the determinant of the (orthogonal) matrix sending the canonical basis of $\RR^d$ to $(n_x, \tau_x^1, \dots, \tau_x^{d-1})$, where $(\tau_x^{1}, \dots, \tau_x^{d-1})$ is an orthonormal basis of $n_x^{\perp}$, has absolute value $1$, we may rewrite $I_{u,x}$ as
\begin{align*}
I_{u,x} &= \int_{-\infty}^0 \int_{\RR^{d-1}} \Big\{ \big( 1+d(\Omega) + (|\vp|^2 + |\vt|^2)^{\frac14} \big)^{\alpha} - (1 + \sqrt{|u|})^{\alpha} \Big\} \\
&\qquad \times f_{\mathcal{N}((1-\rt)\ut, \theta(x) \rt (2-\rt) I_{d-1})} (\vt) \, \\
&\qquad \times  |\vp| \frac{e^{-\frac{|\vp|^2}{2 \theta(x) \rp}}}{\theta(x) \rp} \, e^{-\frac{(1-\rp)|\up|^2}{2 \theta(x) \rp}} \, I_0\Big( \frac{ (1-\rp)^{\frac12} |\up| \vp}{\theta(x) \rp} \Big) \, \d \vt \d \vp,
\end{align*} 
where we (abusively) write $\vp$ for $v \cdot n_x$ to simplify notations. 
We apply the change of variable $\vp \to - \vp$, and, by parity of $I_0$ and definition of the Rice distribution, we find
\begin{align*}
I_{u,x} &= \int_0^{\infty} \int_{\RR^{d-1}} \Big\{ \big( 1+d(\Omega) + (|\vp|^2 + |\vt|^2)^{\frac14} \big)^{\alpha} - (1 + \sqrt{|u|})^{\alpha} \Big\} \\
&\qquad \qquad \times f_{\mathcal{N}((1-\rt)\ut, \theta(x) \rt (2-\rt) I_{d-1})} (\vt) \, f_{\mathrm{Ri}((1-\rp)^{\frac12} |\up|, \theta(x) \rp)}(\vp) \, \d \vp \d \vt.
\end{align*}
We now rewrite $I_{u,x}$ as an expectation:
\begin{align*}
I_{u,x} = \EE \Big[ \big(1 + d(\Omega) + (|X|^2 + |Y|^2)^{\frac14} \big)^{\alpha} \Big] - (1 + \sqrt{|u|})^{\alpha}, 
\end{align*}
with $Y \sim \mathrm{Ri}((1-\rp)^{\frac12} |\up|, \theta(x) \rp)$, $X \sim \mathcal{N}((1-\rt) \ut, \theta(x) \rt (2-\rt) I_{d-1})$. 
Using Proposition \ref{prop:Rice}, we let $\vartheta \in [0,2\pi)$, and consider two random variables independent from everything else (and mutually independent):
\[ Y_1 \sim \mathcal{N}((1-\rp)^{\frac12} |\up| \cos(\vartheta), \theta(x) \rp), \qquad Y_2 \sim \mathcal{N}((1-\rp)^{\frac12} |\up| \sin(\vartheta), \theta(x) \rp). \]
We have $Y \overset{\mathcal{L}}{=} \sqrt{Y_1^2 + Y_2^2}$, so that
\begin{align*}
I_{u,x} = \EE \Big[ \big(1 + d(\Omega) + (Y_1^2 + Y_2^2 + |X|^2)^{\frac14} \big)^{\alpha} \Big] - (1 + \sqrt{|u|})^{\alpha}. 
\end{align*}
This leads to
\begin{align*}
I_{u,x} &=  \EE \Big[ \Big( 1 + d(\Omega) + \Big\{ \bar Y_1^2  + 2 (1-\rp)^{\frac12} \bar Y_1 |\up| \cos(\vartheta)  \\
&\qquad \qquad \qquad \quad  + \bar Y_2^2 + 2 (1-\rp)^{\frac12} \bar Y_2 |\up| \sin(\vartheta) + (1-\rp) |\up|^2  \\
&\qquad \qquad \qquad \quad + |\bar X|^2 + 2 (1 - \rt) \bar X \cdot  \ut  +  (1-\rt)^2 |\ut|^2 \Big\}^{\frac14} \Big)^{\alpha} \Big] \\
&\qquad \quad  - (1 + \sqrt{|u|})^{\alpha}, 
\end{align*}
where 
\begin{align*}
\bar Y_1 &= Y_1 - (1 - \rp)^{\frac12} |\up| \cos(\vartheta) \sim \mathcal{N}(0,\theta(x) \rp), \\ \bar Y_2 &= Y_2 - (1-\rp)^{\frac12} |\up| \sin(\vartheta) \sim \mathcal{N}(0,\theta(x) \rp), \\
\bar X &= X - (1-\rt) \ut \sim \mathcal{N}(0, \theta(x) \rt (2-\rt) I_{d-1}).
\end{align*}
Therefore, using $(1-\rt)^2|\ut|^2 + (1-\rp) |\up|^2 \le \max((1-\rt)^2,(1-\rp)) |u|^2$,
\begin{align}
\label{eq:ineq_tmp_I}
I_{u,x} &\le |u|^{\frac{\alpha}{2}} \Big( \EE \Big[ \Big( \frac{1 + d(\Omega)}{\sqrt{|u|}} + \Big\{ \frac{\bar Y_1^2 + \bar Y_2^2 + |\bar X|^2}{|u|^{2}} + \max((1-\rp), (1-\rt)^2) \\
&\qquad \qquad  + 2\frac{(1-\rp)^{\frac12} |\up| \big( \bar Y_1 \cos(\vartheta) + \bar Y_2  \sin(\vartheta) \big) + (1-\rt) \bar X \cdot \ut} {|u|^2}   \Big\}^{\frac14} \Big)^{\alpha} \Big]	\nonumber \\ 
&\qquad \qquad - \Big( \frac{1}{\sqrt{|u|}} + 1 \Big)^{\alpha} \Big). \nonumber
\end{align}
One can immediately notice that the quantity inside the expectation in \eqref{eq:ineq_tmp_I} is bounded uniformly for all $|u| \ge \Lambda_0$ for some $\Lambda_0 > 0$ large enough, using properties of Gaussian random variables. This converges towards $m := \max((1-\rp), (1-\rt)^2)^{\frac{\alpha}{4}} < 1$ by hypothesis. By dominated convergence theorem, we thus have
\begin{align*}
\lim \limits_{|u| \to \infty} &\EE \Big[ \Big( \frac{1 + d(\Omega)}{\sqrt{|u|}} + \Big\{ \frac{\bar Y_1^2 + \bar Y_2^2 + |\bar X|^2}{|u|^{2}} + \max((1-\rp), (1-\rt)^2) \\
&\qquad   + 2\frac{(1-\rp)^{\frac12} |\up| \big( \bar Y_1 \cos(\vartheta) + \bar Y_2  \sin(\vartheta) \big) + (1-\rt) \bar X \cdot \ut} {|u|^2}   \Big\}^{\frac14} \Big)^{\alpha} \Big] = m. 
\end{align*}
Using this in \eqref{eq:ineq_tmp_I}, since $\frac{1}{\sqrt{|u|}} + 1 \to 1$ as $|u| \to \infty$, we obtain the existence of $\Lambda_0 > 1$ such that for all $|u| \ge \Lambda_0$, using also $\alpha < d+1$, 
\[ I_{u,x} \le |u|^{\tfrac{d+1}{2}} \Big( \frac{m-1}{2} \Big) < 0. \]
Choosing $\Lambda = \Lambda_0 + (\frac{2 C}{1 - m})^{\tfrac{2}{d+1}} > \Lambda_0$, we have, for all $u$ such that $|u| \ge \Lambda$, recalling $m-1 < 0$,
\[ I_{u,x} \le |u|^{\tfrac{d+1}{2}} \Big( \frac{m-1}{2} \Big) \le 2 \frac{C}{1-m} \Big( \frac{m-1}{2} \Big) = - C,\]
and the conclusion follows.
\end{proof}

\begin{proof}[Proof of Proposition \ref{prop:ineq_lyapunov}]
Note first that, since, for all $(x,v) \in G$, $m_\alpha(x,v) = \langle x,v \rangle^{\alpha}$,
\begin{align}
\label{eq:derivative_weight}
 v \cdot \nabla_x m_\alpha(x,v) = (v \cdot \nabla_x \sigma(x,v)) \alpha \langle x,v \rangle^{\alpha - 1} = - \alpha m_{\alpha - 1}. 
 \end{align}

\vspace{.3cm}

\textbf{Step 1.} Let $f \in L^1_{m_\alpha}(G)$. We differentiate the $m_\alpha$-norm of $f$ and use \eqref{eq:derivative_weight}. First, since $n_x$ is the unit outward normal at $x \in \pO$, for $T > 0$, we apply Green's formula to find
\begin{align*}
\frac{d}{dT} \int_G |S_T f| \, m_\alpha \, \d v \d x = \int_G |S_T f| \, (v \cdot \nabla_x m_\alpha) \, \d v \d x - \int_{\Sigma} (v \cdot n_x) \, m_\alpha \, (\gamma |S_T f|) \, \d v \d \zeta(x), 
\end{align*}
where we recall that $\d \zeta$ denotes the induced volume form on $\pO$. We have again, according to Mischler \cite[Corollary 1]{Mischler_Vlasov_1999},
\[ |\gamma S_t| f(x,v) = \gamma |S_t f|(x,v), \qquad \hbox{ a.e. in } (\RR_+^* \times \Sigma_+) \cup (\RR_+^* \times \Sigma_-), \]
hence, we will not distinguish between both values in what follows. We apply \eqref{eq:derivative_weight} to find
\begin{align}
\label{ineq:temp_Lyap}
 \frac{d}{dT} \int_G |S_T f| \, m_\alpha \, \d v \d x &= - \alpha \int_G |S_T f| \, m_{\alpha - 1} \d v \d x \\
 &\qquad - \int_{\Sigma} (v \cdot n_x) \, \gamma |S_T f| \, m_\alpha \, \d v \d \zeta(x). \nonumber
 \end{align}
In the sequel, we let 
\[ B := -\int_{\Sigma} (v\cdot n_x) \, \gamma |S_T f| \, m_\alpha \, \d v \d \zeta(x). \]

\vspace{.3cm}

\textbf{Step 2.} We prove, using Lemma \ref{lemma:Iux}, that there exists $M > 0$ constant such that 
\begin{align}
\label{eq:main_ineq_step_3_Lyapunov}
B \lesssim \int_{\{(x,v) \in \Sigma_+, |v| \le M\}} \gamma_+ |S_T f| \, |\vp| \, \d v \d \zeta(x). 
\end{align}
By definition of $B$,
\begin{align*}
 B &= - \int_{\Sigma_+} \gamma_+ |S_T f| \, |\vp| \, m_\alpha(x,v) \, \d v \d \zeta(x) + \int_{\Sigma_-} \gamma_- |S_T f| \, |\vp| \, m_\alpha(x,v) \, \d v \d \zeta(x)  \\
 &=: -B_1 + B_2,
 \end{align*}
 the last equality standing for a definition of $B_1$ and $B_2$. 
 Using the boundary condition and Tonelli's theorem, it is straightforward to see that
 \begin{align*}
 B_2 = \int_{\Sigma_+} \gamma_+ |S_T f|(u) \, |\up| \Big( \int_{\Sigma_-^x} m_\alpha(x,v) \, |\vp| \, R(u \to v; x) \, \d v \Big) \d u \d \zeta(x).
 \end{align*}
 Set, for all $x \in \pO$, $u \in \Sigma_+^x$, 
 \[ P_{u,x} := \int_{\Sigma_-^x} m_\alpha(x,v) \, |\vp| \, R(u \to v; x) \, \d v. \]
 Note first that, for all $v \in \Sigma_-^x$, $\up \cdot \vp \le 0$ so that, using the definition of $I_0$ \eqref{eq:defI0}, 
 \[ I_0 \Big( \frac{(1-\rp)^{\frac12} \up \cdot \vp}{\theta(x) \rp} \Big) \le e^{-\frac{2(1-\rp)^{\frac12} \up  \cdot \vp}{2\theta(x) \rp}}, \]
 hence, using $\theta(x) \ge \theta_0 > 0$ for all $x \in \pO$, 
 \begin{align*}
 R(u \to v;x) &= \frac{e^{-\frac{|\vt - (1-\rt) \ut|^2}{2 \theta(x) \rt (2-\rt)}}}{(2 \pi \theta(x) \rt (2 - \rt))^{\frac{d-1}{2}}}  \frac{e^{-\frac{|\vp|^2}{2 \theta(x) \rp}}}{\theta(x) \rp} e^{-\frac{(1-\rp)|\up|^2}{2 \theta(x) \rp}} I_0 \Big( \frac{(1 - \rp)^{\frac12} \up \cdot \vp}{ \theta(x) \rp} \Big) \\
 &\le \frac{1}{(2 \pi \theta(x) \rt (2 - \rt))^{\frac{d-1}{2}}} e^{-\frac{|\vt - (1-\rt) \ut|^2}{2 \theta(x) \rt (2-\rt)}} \frac{e^{-\frac{|\vp + (1-\rp)^{\frac12} \up|^2}{2 \theta(x) \rp}}}{\theta(x) \rp} \\ 
 &\le \frac{1}{\theta_0 \rp (2 \pi \theta_0 \rt (2-\rt))^{\frac{d-1}2}} =: C
 \end{align*}
with $C > 0$ constant, where we used the upper bound $1$ for both exponentials. Recall that for all $(x,v) \in \bar G$, $m_{\alpha}(x,v) = (1 + \sigma(x,v) + \sqrt{|v|})^{\alpha}$ and that $d(\Omega)$ denotes the diameter of $\Omega$.
We first have, using that $\sigma(x,v) \le \frac{d(\Omega)}{|v|}$ and that $|\vp| \le |v|$, 
\begin{align*}
\int_{\{v \in \Sigma_-^x, |v| \le 1\}} \hspace{-.8cm} m_\alpha(x,v) \, |\vp| \, R(u \to v; x) \, \d v &\le \int_{\{v \in \Sigma_-^x, |v| \le 1\}} \Big( 2 + \frac{d(\Omega)}{|v|} \Big)^{\alpha} \, |\vp| \, R(u \to v; x) \d v \\
&\le C \int_{\{v \in \Sigma_-^x, |v| \le 1\}} \Big( 2 + \frac{d(\Omega)}{|v|} \Big)^{\alpha} \, |v| \d v \\
&\le C_{\alpha}
\end{align*}
for some constant $C_{\alpha} > 0$ independent of $u$ and $x$. Note that we crucially used that $\alpha < d+1$ to obtain the existence of such finite $C_{\alpha}$ (as can be checked by using an hyperspherical change of variable). On the other hand,
\begin{align*}
\int_{\{v \in \Sigma_-^x, |v| \ge 1\}} &m_\alpha(x,v) \, |\vp| \, R(u \to v; x) \, \d v \\
&\le \int_{\{v \in \Sigma_-^x, |v| \ge 1\}} (1 + d(\Omega) + \sqrt{|v|})^{\alpha} \, |\vp| \, R(u \to v; x) \, \d v \\
&\le \int_{\Sigma_-^x} (1 + d(\Omega) + \sqrt{|v|})^{\alpha} \, |\vp| \, R(u \to v; x) \, \d v.
\end{align*}
Overall, we proved that 
\begin{align}
\label{eq:JU}
 P_{u,x} \le C_{\alpha} + \int_{\Sigma_-^x} (1 + d(\Omega) + \sqrt{|v|})^{\alpha} \, |\vp| \, R(u \to v; x) \, \d v. 
 \end{align}
Using that, for all $(x,u) \in \Sigma_+$, $\int_{\Sigma_-^x} |\vp| \, R(u\to v; x) \, \d v = 1$, and, since 
\[ m_\alpha(x,u) \ge (1+\sqrt{|u|})^{\alpha},\]
 we have
 \begin{align}
 \label{eq:B1}
- B_1 \le - \int_{\Sigma^+} |\up| \, |\gamma_+  S_T f|(x,u)\,  (1+\sqrt{|u|})^{\alpha} \int_{\Sigma_-^x} |\vp| \, R(u \to v; x) \, \d v  \, \d u \d \zeta(x). 
 \end{align}
Gathering \eqref{eq:B1}, \eqref{eq:JU} and the definition of $B$, we find
\begin{align*}
B &\le \int_{\Sigma_+} |\up| \, |\gamma_+ S_T f|(x,u) \, \\
&\quad \times \Big\{ C_{\alpha} + \int_{\Sigma_-^x} \Big[ (1 + d(\Omega) + \sqrt{|v|})^{\alpha} - (1 + \sqrt{|u|})^{\alpha} \Big] \, |\vp| \, R(u \to v; x) \, \d v \Big\} \d u \d \zeta(x) \\
&\le \int_{\Sigma_+} |\up| \, |\gamma_+ S_T f|(x,u) \, \big( C_{\alpha} + I_{u,x} \big) \d u \d \zeta(x),
\end{align*}
where $I_{u,x}$ is defined as in Lemma \ref{lemma:Iux}. Splitting $\Sigma_+^x$ as
\[ \Sigma_+^x = \{u \in \Sigma_+^x: |u| < \Lambda\} \cup \{u \in \Sigma_+^x: |u| \ge \Lambda\} \]
with $\Lambda > 0$ given by Lemma \ref{lemma:Iux} applied with $C = C_{\alpha}$, we find that 
\[ \int_{\{(x,u) \in \Sigma_+, |u| \ge \Lambda\}} |\up| \, |\gamma_+ S_T f|(x,u) \, \big( C_{\alpha} + I_{u,x} \big) \d u \d \zeta(x) \le 0, \]
leading to 
\begin{align}
\label{eq:tmp_B}
B &\le \int_{\pO} \int_{\{u \in \Sigma_+^x, |u| \le \Lambda\}} |\up| \, |\gamma_+ S_T f|(x,u) \Big(C_{\alpha} + I_{u,x} \Big) \d u \d \zeta(x)  \nonumber  \\
&\le \int_{\pO} \int_{\{u \in \Sigma_+^x, |u| \le \Lambda\}} |\up| \, |\gamma_+ S_T f|(x,u) \nonumber \\
&\qquad \qquad \times \Big( C_{\alpha} + \int_{\Sigma_-^x} (1+d(\Omega) + \sqrt{|v|})^{\alpha} \, |\vp| \, R(u\to v;x) \, \d v \Big) \d u \d \zeta(x).  
\end{align}
We claim that, for all $x \in \pO$, $u \in \Sigma_+^x$ with $|u| \le \Lambda$,
\[ \int_{\Sigma_-^x} (1 + d(\Omega) + \sqrt{|v|})^{\alpha} \, |\vp| \, R(u \to v; x) \d v \in (0,\infty). \]
This can be seen again by using probability theory. We write this integral as
\[ \EE \Big[ \Big(1 + d(\Omega) + \big(|X|^2 + |Y|^2\big)^{\frac14} \Big)^{\alpha}\Big], \]
for $Y \sim \mathrm{Ri}((1-\rp)^{\frac12}|\up|, \theta(x) \rp)$ and $X \sim \mathcal{N}((1-\rt) \ut, \theta(x) \rt (2 - \rt) I_{d-1})$ two independent random variables. Using Proposition \ref{prop:Rice}, we have, for any $\hat{\beta} \in [0,2\pi)$, 
\[ Y \overset{\mathcal{L}}{=} \sqrt{Y_1^2 + Y_2^2} \quad \text{  with } \]
\[ Y_1 \sim \mathcal{N}((1-\rp)^{\frac12} |\up| \cos(\hat{\beta}), \theta(x) \rp), \qquad Y_2 \sim \mathcal{N}((1-\rp)^{\frac12} |\up| \sin(\hat{\beta}), \theta(x) \rp) \] 
two random variables independent from everything else. By standard properties of the moments of Gaussian random variables, the claim follows.  

Using that $x \to n_x$ and $x \to \theta(x)$ are continuous, we have $(x,u) \to R(u \to v; x)$ continuous, hence, by compactness of $\{(x,u) \in \Sigma_+, |u| \le \Lambda\}$ and continuity under the integral sign, there exists $C'_{\alpha, \Lambda} >0$ such that for all $(x,u) \in \Sigma_+$ with $|u| \le \Lambda$, 
\[ \int_{\Sigma_-^x} (1 + d(\Omega) + \sqrt{|v|})^{\alpha} \, |\vp| \, R(u \to v; x) \d v \le C'_{\alpha,\Lambda}. \]
Using this in \eqref{eq:tmp_B}, we have
\begin{align*}
B \le \big(C_{\alpha} + C'_{\alpha, \Lambda} \big) \int_{\pO} \int_{\{ v \in \Sigma_+^x, |v| \le \Lambda\}} \gamma_+ |S_T f|(x,v) \, |\vp| \d v \d \zeta(x). 
\end{align*}
We plug this inequality into \eqref{ineq:temp_Lyap} and conclude that, for $C_{\alpha,\Lambda} = C_{\alpha} + C'_{\alpha,\Lambda} > 0$,
\begin{align}
\label{eq:conclusion_Step_2_lyapunov}
\frac{d}{dT} &\int_G |S_T f| \, m_\alpha(x,v) \, \d v \d x \\
&\le - \alpha \int_G |S_T f| \, m_{\alpha - 1}(x,v) \, \d v \d x + C_{\alpha,\Lambda} \int_{\{(x,v) \in \Sigma_+, |v| \le \Lambda\}} \gamma_+ |S_T f|(x,v) \, |\vp| \, \d v \d \zeta(x). \nonumber
\end{align}  

\vspace{.3cm}

\textbf{Step 3.} We use the conclusion of Step 2, \eqref{eq:conclusion_Step_2_lyapunov}, and Lemma \ref{lemma:control_flux} to conclude the proof of Proposition \ref{prop:ineq_lyapunov}. 

We integrate \eqref{eq:conclusion_Step_2_lyapunov} between $0$ and $T > 0$ to find
\begin{align*}
\|S_T f\|_{m_\alpha} &+ \alpha \int_0^T \|S_s f\|_{m_{\alpha - 1}} \, \d s \\
&\le \|f\|_{m_\alpha} 
 + C_{\alpha,\Lambda} \int_0^T \int_{\{(x,v) \in \Sigma_+, |v| \le \Lambda\}} \gamma_+ |S_T f| (x,v) |\vp| \d v \d \zeta(x).
\end{align*}
Applying Lemma \ref{lemma:control_flux} and setting $b := C_{\alpha,\Lambda} C_\Lambda > 0$ where $C_\Lambda > 0$ is given by the lemma, we find, 
\begin{align*}
\|S_T f\|_{m_\alpha} + \alpha \int_0^T \|S_s f\|_{m_{\alpha - 1}} \, \d s \le \|f\|_{m_\alpha} +b (1+T) \|f\|_{L^1},
\end{align*}
as claimed. 
\end{proof}

\section{Doeblin-Harris condition}
\label{section:DH}

Recall that $\Omega$ is a $C^2$ bounded domain. In this section, we prove the Doeblin-Harris condition, Theorem \ref{thm:Doeblin-Harris}, by adapting the argument of \cite{Bernou_Transport_Semigroup_2020} to the present case. We also simplify slightly some steps at the end of the demonstration. For any two points $x, y \in \pO$, we write 
\[ ]x,y[ := \{tx + (1-t)y, t \in ]0,1[ \}. \]

\begin{defi}
\label{defi:communique}
For $(x,y) \in (\pO)^2$, we write $x \leftrightarrow y$ and say that $x$ and $y$ see each other if $]x,y[ \subset \Omega$, $n_x \cdot (y-x) > 0$ and $n_y \cdot (x - y) > 0$. 
\end{defi}

We will crucially use the following result on $C^1$ bounded domain given by Evans.

\begin{prop}[Proposition 1.7 in \cite{Evans_1999}]
\label{prop:evans}
For all $C^1$ bounded domain $C$, there exist an integer $P$ and a finite set $\Delta' \subset \partial C$ for which the following holds: for all $z', z'' \in \partial C$, there exist $z_0, \dots, z_P$ with $z' = z_0$, $z'' = z_P$, $\{z_1,\dots,z_{P-1}\} \subset \Delta'$ and $z_k \leftrightarrow z_{k+1}$ for $0 \le k \le P-1$. 
\end{prop}

We now state the main result of this section. Recall that for all $(x,v) \in \bar{G}$, we have $\langle x,v \rangle = (1 + \sigma(x,v) + \sqrt{|v|})$ and that $(S_t)_{t \ge 0}$ denotes the semigroup associated to \eqref{eq:main_pb} as introduced in Section \ref{section:setting}.

\begin{thm}
\label{thm:Doeblin-Harris}
For any $\Lambda \ge 2$, there exist $T(\Lambda) > 0$ and a non-negative measure $\nu$ on $G$ with $\nu \not \equiv 0$ such that for all $(x,v) \in G$, for all $f_0 \in L^1(G)$, $f_0 \ge 0$, 
\begin{align}
\label{eq:Doeblin-Harris}
S_{T(\Lambda)} f_0(x,v) \ge \nu(x,v) \int_{\{(y,w) \in G, \langle y, w \rangle \le \Lambda\}} f_0(y,w) \, \d y \d w.
\end{align}
Moreover, $\nu$ satisfies $\langle \nu \rangle \le 1$ and there exists $\kappa > 0$ such that for all $\Lambda \ge 2$, $T(\Lambda) = \kappa \Lambda$.  
\end{thm}

\begin{proof}
We only treat the case $d = 3$, as the case $d = 2$ follows from similar (easier) computations. For all $t > 0$, $(x,v) \in \bar{G}$, we write $f(t,x,v) = S_t f_0(x,v)$. For the sake of simplicity we simply write $f(t,x,v)$ for $\gamma f(t,x,v)$ for $(t,x,v) \in \RR_+ \times \Sigma$.

\vspace{.3cm}

\textbf{Step 1.}
We let $(t,x,v) \in (0,\infty) \times G$ and compute a first lower-bound for $f(t,x,v)$. Recall the definitions of $\sigma$ from \eqref{eq:def_sigma} and $q$ from \eqref{eq:def_q}. From the characteristics method, we have
\[ f(t,x,v) = f_0(x - tv,v) \mathbf{1}_{\{t < \sigma(x,-v)\}} + f(t - \sigma(x,-v),q(x,-v),v) \mathbf{1}_{\{t \ge \sigma(x,-v)\}}. \]
Set $y_0 = q(x,-v)$, $\tau_0 = \sigma(x,-v)$. We have, using the boundary condition and the characteristics of the free-transport equation, along with the positivity of $f_0$, 
\begin{align*}
f(t,x,v) &\ge \mathbf{1}_{\{\tau_0 \le t\}} f(t - \tau_0, y_0,v) \\
&\ge \mathbf{1}_{\{\tau_0 \le t\}} \int_{\Sigma_+^{y_0}} f(t- \tau_0,y_0,v_0) \, |v_0 \cdot n_{y_0}| \, R(v_0 \to v; y_0) \, \d v_0 \\
&\ge \mathbf{1}_{\{\tau_0 \le t\}} \int_{\Sigma_+^{y_0}} f(t- \tau_0 - \sigma(y_0,-v_0),q(y_0,-v_0),v_0) \, \mathbf{1}_{\{\tau_0 + \sigma(y_0,-v_0) \le t\}} \,\\
&\qquad \times |v_0 \cdot n_{y_0}| \, R(v_0 \to v; y_0) \, \d v_0 \\
&\ge \mathbf{1}_{\{\tau_0 \le t\}} \int_{\Sigma_+^{y_0}} \mathbf{1}_{\{\tau_0 + \sigma(y_0,-v_0) \le t\}} \,  |v_0 \cdot n_{y_0}| \, R(v_0 \to v; y_0) \, \\
&\qquad \times \int_{\Sigma_+^{q(y_0,-v_0)}} |v_1 \cdot n_{q(y_0,-v_0)}| \, R(v_1 \to v_0; q(y_0,-v_0)) \\ &\qquad  \times f(t- \tau_0 - \sigma(y_0,-v_0),q(y_0,-v_0),v_1) \, \d v_1 \d v_0.
\end{align*}

We write $v_0$ in spherical coordinates $(r,\phi,\vartheta) \in \RR_+ \times [-\pi,\pi] \times [0,\pi]$ in the space directed by the vector $n_{y_0}$. We let $u = u(\phi,\vartheta)$ be the unit vector corresponding to the direction of $v_0$. The condition $v_0 \cdot n_{y_0} > 0$ is equivalent to $\phi \in (-\frac{\pi}{2}, \frac{\pi}{2})$, and we obtain from the previous inequality, using also that $q(y_0,-v_0) = q(y_0,-u)$ as this point is independent of $|v_0| = r$,
\begin{align*}
f(t,x,v) &\ge \mathbf{1}_{\{\tau_0 \le t\}} \int_0^{\infty} \int_{-\frac{\pi}{2}}^{\frac{\pi}{2}} \int_0^{\pi} \mathbf{1}_{\{\tau_0 + \frac{\sigma(y_0,-u)}{r} \le t\}} \, |u \cdot n_{y_0}| \, \sin(\vartheta) \, r^3 \, R(ru \to v; y_0) \\
&\quad \times \int_{\Sigma_+^{q(y_0,-u)}} |v_1 \cdot n_{q(y_0,-u)}| \, R(v_1 \to ru; q(y_0,-u)) \\
&\quad \times f(t-\tau_0 - \frac{\sigma(y_0,-u)}{r}, q(y_0,-u),v_1) \, \d v_1 \d \vartheta \d \phi \d r.
\end{align*}
We follow \cite{Bernou_Transport_Semigroup_2020} and use the change of variable $(y_1,\tau_1) = (q(y_0,-u),\sigma(y_0,-ru))$. The inverse of the determinant of the Jacobian matrix was derived by Esposito, Guo, Kim and Marra \cite[Lemma 2.3]{esposito_non-isothermal_2013} and is given, in the case where $y_1 \leftrightarrow y_0$, by
\begin{align*}
\frac{\tau_1^3 r \sin(\vartheta) |\partial_3 \xi(y_1)|}{|u \cdot n_{y_1}| |\nabla_x \xi(y_1)|},
\end{align*}
where $\xi$ is the $C^1$ function that locally parametrizes $\Omega$, hence 
\[ \Omega = \{y \in \RR^d: \xi(y) < 0\}, \] with the further assumption, which can be made without loss of generality, that $\partial_3 \xi(y_1) \ne 0$. Finally, $u$ is the unit vector giving the direction going from $y_1$ to $y_0$, hence
\begin{align*}
u = \frac{y_0 - y_1}{|y_0 - y_1|}, \qquad r = \frac{|y_0-y_1|}{\tau_1}.
\end{align*}
Setting, for $a \in \pO$,
\[ U_a = \{y \in \pO, y \leftrightarrow a\}, \]
we obtain from the previous inequality, by applying this change of variable,
\begin{align*}
f(t,x,v) 
&\ge \mathbf{1}_{\{\tau_0 \le t\}} \int_0^{t - \tau_0} \int_{U_{y_0}} |u \cdot n_{y_0}| \, |u \cdot n_{y_1}| \, \frac{|y_1 - y_0|^2}{\tau_1^5} \,  \, \\
&\quad \times R \Big( \frac{y_0 - y_1}{\tau_1} \to v; y_0 \Big) \, \frac{|\nabla_x \xi(y_1)|}{|\partial_3 \xi (y_1)|} \\
&\quad \times  \int_{\Sigma_+^{y_1}} f(t - \tau_0 - \tau_1, y_1,v_1) \, |v_1 \cdot n_{y_1}| R\Big( v_1 \to \frac{y_0 - y_1}{\tau_1};y_1) \, \d v_1 \d y_1 \d \tau_1 \\
&\ge 
\mathbf{1}_{\{\tau_0 \le t\}} \int_0^{t - \tau_0} \int_{U_{y_0}}  \frac{|(y_0-y_1) \cdot n_{y_0}| |(y_0 - y_1) \cdot n_{y_1}|}{\tau_1^5}  \, R \Big( \frac{y_0 - y_1}{\tau_1} \to v; y_0 \Big) \\
&\quad \times \int_{\Sigma_+^{y_1}}   |v_1 \cdot n_{y_1}| R\Big( v_1 \to \frac{y_0 - y_1}{\tau_1};y_1 \Big) \, \mathbf{1}_{\{\tau_0 + \tau_1 + \sigma(y_1,-v_1) \le t\}} \\
&\quad \times \, f(t - \tau_0 - \tau_1 - \sigma(y_1,-v_1), q(y_1,-v_1),v_1) \, \d v_1 \d \zeta(y_1) \d \tau_1,
\end{align*}
where we used again the characteristics of the free-transport equation, and with $\d \zeta$ the surface measure of $\pO$, which is given by $\d \zeta(y) = \frac{|\nabla_x \xi(y)|}{|\partial_3 \xi (y)|} \d y$ for any $y \in \pO$. 
We use one last time the boundary condition to obtain
\begin{align*}
f(t,x,v) 
&\ge 
\mathbf{1}_{\{\tau_0 \le t\}} \int_0^{t - \tau_0} \int_{U_{y_0}}  \frac{|(y_0-y_1) \cdot n_{y_0}| |(y_0 - y_1) \cdot n_{y_1}|}{\tau_1^5}  \, R \Big( \frac{y_0 - y_1}{\tau_1} \to v; y_0 \Big) \\
&\quad \times \int_{\Sigma_+^{y_1}}   |v_1 \cdot n_{y_1}| R\Big( v_1 \to \frac{y_0 - y_1}{\tau_1};y_1 \Big) \, \mathbf{1}_{\{\tau_0 + \tau_1 + \sigma(y_1,-v_1) \le t\}} \\
&\quad \times \Big( \int_{\Sigma_+^{q(y_1,-v_1)}} |v_2 \cdot n_{q(y_1,-v_1)}| \, R(v_2 \to v_1; q(y_1,-v_1))  \\ 
&\quad \times \, f(t - \tau_0 - \tau_1 - \sigma(y_1,-v_1), q(y_1,-v_1),v_2) \, \d v_2 \Big) \, \d v_1 \d \zeta(y_1) \d \tau_1.
\end{align*}

\vspace{.3cm}

\textbf{Step 2.} We iterate the method of Step 1 $P-2$ times and make a change of variable to recover an integral over a subset of $G$. 

Let $P \in \mathbb{Z}^+$ be given by Proposition \ref{prop:evans}. We repeat the previous computation $P-2$ times to find
\begin{align*}
f(t,x,v) 
&\ge 
\mathbf{1}_{\{\tau_0 \le t\}} \int_0^{t-\tau_0} \int_{U_{y_0}} \frac{|(y_1 - y_0) \cdot n_{y_0}| |(y_1 - y_0) \cdot n_{y_1}|}{\tau_1^5} \, R\Big( \frac{y_0 - y_1}{\tau_1} \to v; y_0 \Big) \\
&\quad \times \int_0^{t - \tau_0 - \tau_1} \int_{U_{y_1}} \frac{|(y_2 - y_1) \cdot n_{y_1}||(y_2 - y_1) \cdot n_{y_2}|}{\tau_2^5} \, R \Big( \frac{y_1 - y_2}{\tau_2} \to \frac{y_0 - y_1}{\tau_1}; y_1 \Big) \\
&\quad \times \dots \\
&\quad \times \int_0^{t - \sum_{i=0}^{P-1} \tau_i} \int_{U_{y_{P-1}}} \frac{|(y_P - y_{P-1})\cdot n_{y_{P-1}}| |(y_P - y_{P-1}) \cdot n_{y_P}|}{\tau_P^5}  \\
&\quad \times R \Big( \frac{y_{P-1} - y_P}{\tau_P} \to \frac{y_{P-2} - y_{P-1}}{\tau_{P-1}}; y_{P-1} \Big) \, \\
&\quad \times \int_{\Sigma_+^{y_P}} |v_P \cdot n_{y_P}|  R \Big( v_P \to \frac{y_{P-1} - y_P}{\tau_P} ; y_P \Big) \,  \\
&\quad \times f \Big(t - \sum_{i=0}^P \tau_i, y_P, v_P \Big) \, \d v_P \d \zeta(y_P) \d \tau_P \dots \d \zeta(y_1) \d \tau_1.
\end{align*}
On the set $\{t \ge \sum_{i=0}^P \tau_i\}$, by positivity and using the method of characteristics, we have
\[ f \Big(t - \sum_{i=0}^P \tau_i, y_P,v_P \Big) \ge f_0\Big(y_P - \Big(t - \sum_{i=0}^P \tau_i\Big) v_P, v_P \Big) \mathbf{1}_{\{t - \sum_{i=0}^P \tau_i - \sigma(y_P, -v_P) \le 0 \}}, \]
hence, we can lower-bound the previous inequality in the following way:
\begin{align*}
f(t&,x,v) 
\ge 
\mathbf{1}_{\{\tau_0 \le t\}} \int_0^{t-\tau_0} \int_{U_{y_0}} \frac{|(y_1 - y_0) \cdot n_{y_0}| |(y_1 - y_0) \cdot n_{y_1}|}{\tau_1^5} \, R\Big( \frac{y_0 - y_1}{\tau_1} \to v; y_0 \Big) \\
&\quad \times \int_0^{t - \tau_0 - \tau_1} \int_{U_{y_1}} \frac{|(y_2 - y_1) \cdot n_{y_1}||(y_2 - y_1) \cdot n_{y_2}|}{\tau_2^5} \, R \Big( \frac{y_1 - y_2}{\tau_2} \to \frac{y_0 - y_1}{\tau_1}; y_1 \Big) \\
&\quad \times \dots \\
&\quad \times \int_0^{t - \sum_{i=0}^{P-1} \tau_i} \int_{U_{y_{P-1}}} \frac{|(y_P - y_{P-1})\cdot n_{y_{P-1}}| |(y_P - y_{P-1}) \cdot n_{y_P}|}{\tau_P^5}  \\
&\quad \times R \Big( \frac{y_{P-1} - y_P}{\tau_P} \to \frac{y_{P-2} - y_{P-1}}{\tau_{P-1}}; y_{P-1} \Big) \, \\
&\quad \times \int_{\Sigma_+^{y_P}} |v_P \cdot n_{y_P}| \, R \Big( v_P \to \frac{y_{P-1} - y_P}{\tau_P} ; y_P \Big) \, f_0\Big(y_P - \Big(t - \sum_{i=0}^P \tau_i\Big) v_P, v_P \Big) \, \\
&\quad \times  \mathbf{1}_{\{t - \sum_{i=0}^P \tau_i - \sigma(y_P, -v_P) \le 0 \}} \, \d v_P \d \zeta(y_P) \d \tau_P \dots \d \zeta(y_1) \d \tau_1.
\end{align*}
We now set $z = \psi(y_P,\tau_P) = y_P - v_P (t - \sum_{i=0}^P \tau_i)$, and compute the result of this change of variable from $(y_P,\tau_P)$ to $z$. The map $\psi$ is a $C^1$ diffeomorphism with
\begin{align*}
\psi:& \Big\{ (y_P,\tau_P) \in \pO \times \RR_+ : \sigma(y_P,-v_P) > t - \sum_{i=0}^P \tau_i > 0, y_P \leftrightarrow y_{P-1} \Big\} \\
&\to \Big \{z \in \Omega: q(z,v_P) \leftrightarrow y_{P-1}, \sigma(z,v_P) + \sum_{i=0}^{P-1} \tau_i \le t \Big\}.
\end{align*} 
 With this change of variable, $y_P = q(z,v_P)$. Moreover, $t - \sum_{i=0}^P \tau_i = \sigma(z,v_P)$ by definition of $z$, so that 
 \[ \tau_P = t - \sum_{i=0}^{P-1} \tau_i - \sigma(z,v_P). \]
 The inverse of the Jacobian of $\psi$ is $|v_P \cdot n_{y_P}|$, see again Esposito et al. \cite[Lemma 2.3]{esposito_non-isothermal_2013}. Therefore,
\begin{align*}
f(t&,x,v) \ge 
\mathbf{1}_{\{\tau_0 \le t\}} \int_0^{t-\tau_0} \int_{U_{y_0}} \frac{|(y_1 - y_0) \cdot n_{y_0}| |(y_1 - y_0) \cdot n_{y_1}|}{\tau_1^5} \, R\Big( \frac{y_0 - y_1}{\tau_1} \to v; y_0 \Big) \\
&\quad \times \int_0^{t - \tau_0 - \tau_1} \int_{U_{y_1}} \frac{|(y_2 - y_1) \cdot n_{y_1}||(y_2 - y_1) \cdot n_{y_2}|}{\tau_2^5} \, R \Big( \frac{y_1 - y_2}{\tau_2} \to \frac{y_0 - y_1}{\tau_1}; y_1 \Big) \\
&\quad \times \dots \\
&\quad \times \int_0^{t - \sum_{i=0}^{P-2} \tau_i} \int_{U_{y_{P-2}}} \frac{|(y_{P-2} - y_{P-1})\cdot n_{y_{P-1}}| |(y_{P-2} - y_{P-1}) \cdot n_{y_{P-2}}|}{\tau_{P-1}^5}  \\
&\quad \times R \Big( \frac{y_{P-2} - y_{P-1}}{\tau_{P-1}} \to \frac{y_{P-3} - y_{P-2}}{\tau_{P-2}} ; y_{P-2} \Big) \\
& \quad \times \Big\{ \int_G \frac{|(y_{P-1} - q(z,v_P)) \cdot n_{q(z,v_P)}| |(y_{P-1} - q(z,v_P)) \cdot n_{y_{P-1}}|}{\big(t - \sum_{i=0}^{P-1} \tau_i - \sigma(z,v_P)\big)^5} \, \\
& \quad \times  R \Big( \frac{y_{P-1} - q(z,v_P)}{t - \sum_{i=0}^{P-1} \tau_i - \sigma(z,v_P)} \to \frac{y_{P-2} - y_{P-1}}{\tau_{P-1}}; y_{P-1} \Big) \\
&\quad \times R \Big( v_P \to  \frac{y_{P-1} - q(z,v_P)}{t - \sum_{i=0}^{P-1} \tau_i - \sigma(z,v_P)}; q(z,v_P) \Big) \,  \mathbf{1}_{\{q(z,v_P) \leftrightarrow y_{P-1} \}} \, \\ 
&\quad \times \mathbf{1}_{\{ \sigma(z,v_P) + \sum_{i=0}^{P-1} \tau_i \le t\}}  f_0(z,v_P) \d v_P \d z \Big\} \d \zeta(y_{P-1}) \, \d \tau_{P-1} \dots \d \zeta(y_1) \d \tau_1.
\end{align*}
Using Tonelli's theorem, we then have
\begin{align}
\label{eq:DH_final_step2}
f(t,&x,v) \ge \mathbf{1}_{\{\tau_0 \le t\}} \int_G f_0(z,v_P) \\
&\times \int_0^{t-\tau_0} \int_{U_{y_0}} \frac{|(y_1 - y_0) \cdot n_{y_0}| |(y_1 - y_0) \cdot n_{y_1}|}{\tau_1^5} \, R\Big( \frac{y_0 - y_1}{\tau_1} \to v; y_0 \Big) \nonumber \\
&\quad \times \int_0^{t - \tau_0 - \tau_1} \int_{U_{y_1}} \frac{|(y_2 - y_1) \cdot n_{y_1}||(y_2 - y_1) \cdot n_{y_2}|}{\tau_2^5} \, R \Big( \frac{y_1 - y_2}{\tau_2} \to \frac{y_0 - y_1}{\tau_1}; y_1 \Big) \nonumber \\
&\quad \times \dots \nonumber \\
&\quad \times \int_0^{t - \sum_{i=0}^{P-2} \tau_i} \int_{U_{y_{P-2}}} \frac{|(y_{P-2} - y_{P-1})\cdot n_{y_{P-1}}| |(y_{P-2} - y_{P-1}) \cdot n_{y_{P-2}}|}{\tau_{P-1}^5}  \nonumber \\
&\quad \times R \Big( \frac{y_{P-2} - y_{P-1}}{\tau_{P-1}} \to \frac{y_{P-3} - y_{P-2}}{\tau_{P-2}} ; y_{P-2} \Big) \nonumber \\
& \quad \times \frac{|(y_{P-1} - q(z,v_P)) \cdot n_{q(z,v_P)}| |(y_{P-1} - q(z,v_P)) \cdot n_{y_{P-1}}|}{\big(t - \sum_{i=0}^{P-1} \tau_i - \sigma(z,v_P)\big)^5} \, \nonumber \\
& \quad \times  R \Big( \frac{y_{P-1} - q(z,v_P)}{t - \sum_{i=0}^{P-1} \tau_i - \sigma(z,v_P)} \to \frac{y_{P-2} - y_{P-1}}{\tau_{P-1}}; y_{P-1} \Big) \nonumber \\
&\quad \times R \Big( v_P \to  \frac{y_{P-1} - q(z,v_P)}{t - \sum_{i=0}^{P-1} \tau_i - \sigma(z,v_P)}; q(z,v_P) \Big) \, \mathbf{1}_{\{q(z,v_P) \leftrightarrow y_{P-1} \}} \, \nonumber \\ 
&\quad \times \mathbf{1}_{\{ \sigma(z,v_P) + \sum_{i=0}^{P-1} \tau_i < t\}}  \, \d \zeta(y_{P-1}) \, \d \tau_{P-1} \dots \d \zeta(y_1) \d \tau_1 \d v_P \d z. \nonumber
\end{align}

\vspace{.3cm}

\textbf{Step 3.} We choose the value of $t$ and control the time integrals in \eqref{eq:DH_final_step2}. 
Let $\Lambda > 2$ and set $t = (2P+2) \Lambda$, $\tau_0 \in (\Lambda,2\Lambda)$, i.e for all $(x,v) \in G$ such that $\sigma(x,-v) \not \in (\Lambda,2\Lambda)$, we simply set $\nu(x,v) = 0$. Note that, for any $\Lambda > 0$, one can find a couple $(x,v) \in G$ such that $\sigma(x,-v) \in (\Lambda,2\Lambda)$, which also implies $|v| \le \frac{d(\Omega)}{\Lambda}$. 

For all $i \in \{1,\dots,P-1\}$, we lower bound the integral with respect to $\tau_i$ by an integral over $(\Lambda,2\Lambda)$. We also lower bound the integral with respect to $(z,v_P)$ by an integral over $D_\Lambda = \{(z,v_P) \in G: \langle z,v_P \rangle \le \Lambda\}$, which is not empty since $\Lambda > 2$. Note that, on $D_\Lambda$, 
\begin{align}
\label{eq:prop_D_lambda}
 \sigma(z,v_P) \le \Lambda, \qquad |v_P| \le \Lambda^2.
 \end{align}

With $\tau_0,\dots,\tau_{P-1} \in (\Lambda,2\Lambda)$, $\sigma(z,v_P) \le \Lambda$ and $t = (2P  +2)\Lambda$, we have first
\begin{subequations}
\begin{align}
\label{eq:t-quanti-1}
(2P + 2) \Lambda - 2P\Lambda - \Lambda = \Lambda &\le t - \sum_{i=0}^{P-1} \tau_i - \sigma(z,v_P) \\
\label{eq:t-quanti-2}
t - \sum_{i=0}^{P-1} \tau_i - \sigma(z,v_P) &\le (2P+2)\Lambda - P\Lambda = (P+2)\Lambda, 
\end{align}
\end{subequations}
thus, with those choices,
\[ \mathbf{1}_{\{ \sum_{i=0}^{P-1} \tau_i + \sigma(z,v_P) \le t\}} = 1. \]
Moreover, recalling that for all $i \in \{1,\dots,P-1\}$ the integration interval for $\tau_i$ in \eqref{eq:DH_final_step2} is $[0,t-\sum_{j=0}^{i-1} \tau_j]$, and since 
\[ t - \sum_{j=0}^{i-1} \tau_j \ge (2P+2)\Lambda - 2i\Lambda = 2\Lambda + 2(P-i)\Lambda \ge 2\Lambda, \]
the lower bound detailed above using an integral over $[\Lambda,2\Lambda]$ for $\tau_i$ is legitimate. 

Applying those lower bounds, we find
\begin{align}
\label{eq:dh_step_3_R}
f(t,x,v) \ge &\mathbf{1}_{\{\tau_0 \in [\Lambda,2\Lambda]\}} \int_{D_\Lambda} f_0(z,v_P) \\
&\quad \times \int_\Lambda^{2\Lambda} \int_{U_{y_0}} \frac{|(y_1 - y_0) \cdot n_{y_0}| |(y_1 - y_0) \cdot n_{y_1}|}{\tau_1^5} \, R\Big( \frac{y_0 - y_1}{\tau_1} \to v; y_0 \Big)  \nonumber \\
&\quad \times \int_\Lambda^{2\Lambda} \int_{U_{y_1}} \frac{|(y_2 - y_1) \cdot n_{y_1}||(y_2 - y_1) \cdot n_{y_2}|}{\tau_2^5} \, R \Big( \frac{y_1 - y_2}{\tau_2} \to \frac{y_0 - y_1}{\tau_1}; y_1 \Big) \nonumber \\
&\quad \times \dots \nonumber \\
&\quad \times \int_\Lambda^{2\Lambda} \int_{U_{y_{P-2}}} \frac{|(y_{P-2} - y_{P-1})\cdot n_{y_{P-1}}| |(y_{P-2} - y_{P-1}) \cdot n_{y_{P-2}}|}{\tau_{P-1}^5} \nonumber \\
&\quad \times R \Big( \frac{y_{P-2} - y_{P-1}}{\tau_{P-1}} \to \frac{y_{P-3} - y_{P-2}}{\tau_{P-2}} ; y_{P-2} \Big) \nonumber \\
& \quad \times \frac{|(y_{P-1} - q(z,v_P)) \cdot n_{q(z,v_P)}| |(y_{P-1} - q(z,v_P)) \cdot n_{y_{P-1}}|}{\big(t - \sum_{i=0}^{P-1} \tau_i - \sigma(z,v_P)\big)^5} \, \nonumber \\
& \quad \times  R \Big( \frac{y_{P-1} - q(z,v_P)}{t - \sum_{i=0}^{P-1} \tau_i - \sigma(z,v_P)} \to \frac{y_{P-2} - y_{P-1}}{\tau_{P-1}}; y_{P-1} \Big)  \nonumber \\
&\quad \times R \Big( v_P \to  \frac{y_{P-1} - q(z,v_P)}{t - \sum_{i=0}^{P-1} \tau_i - \sigma(z,v_P)}; q(z,v_P) \Big) \, \mathbf{1}_{\{q(z,v_P) \leftrightarrow y_{P-1} \}} \nonumber  \\
&\quad \times \d \zeta(y_{P-1}) \, \d \tau_{P-1} \dots \d \zeta(y_1) \d \tau_1 \d v_P \d z. \nonumber
\end{align}
Note that, for all $u,v \in \RR^d$, $x \in \pO$, with $ u \cdot n_x > 0$, $v \cdot n_x < 0$, 
\begin{align}
\label{eq:mino_R}
 R(u\to v;x) \ge \frac{1}{\theta(x) \rp (2 \pi \theta(x) \rt (2 - \rt))^{\frac{d-1}{2}}} e^{-\frac{|\vt - (1-\rt) \ut|^2}{2 \theta(x) \rt (2-\rt)}} e^{-\frac{||\vp| - (1-\rp)^{\frac12} |\up||^2}{2 \theta(x) \rp}}, 
 \end{align}
where we used that $I_0(\frac{(1-\rp)^{\frac{1}{2}} \up \cdot \vp}{\theta(x) \rp}) \ge e^{-\frac{(1-\rp)^{\frac12} |\up| |\vp|}{\theta(x) \rp}}$, and by continuity of the right-hand side of \eqref{eq:mino_R}, using that $x \to \theta(x)$ and $x \to n_x$ are continuous, we obtain, by a compactness argument, that for all $M_1, M_2 > 0$,
\begin{align*}
\inf_{x \in \pO, |u| \le M_1, |v| \le M_2} R(u \to v; x) \ge c_{M_1,M_2} > 0
\end{align*}
with $c_{M_1,M_2}$ depending only on $M_1, M_2$.
We now study the arguments of $R$ inside the integrals of \eqref{eq:dh_step_3_R}. We have
\begin{enumerate}
\item $|v_P| \le \Lambda^2$, by \eqref{eq:prop_D_lambda},
\item for all $i \in \{0,\dots,P-2\}$, $|\frac{y_i - y_{i+1}}{\tau_{i+1}}| \le \frac{d(\Omega)}{\Lambda}$,
\item $|\frac{y_{P-1} - q(z,v_P)}{t - \sum_{i=0}^{P-1} \tau_i - \sigma(z,v_P)} | \le \frac{d(\Omega)}{\Lambda}$, 
\end{enumerate}
where the last inequality uses \eqref{eq:t-quanti-1}. 
Finally we introduce a measure $\underline{R}$ on $\Sigma_-$ such that for all $(y,v) \in \Sigma_-$,
\[ \underline{R}(y,v) = \inf_{|u| \le \frac{d(\Omega)}{\Lambda}, u \cdot n_y > 0} R(u \to v; y). \]
Note that for all $(y,v) \in \Sigma_-$, $\underline{R}(y,v) > 0$. A straightforward application of those bounds, along with the definition of $c_{\cdot,\cdot}$ leads to 
\begin{align*}
f(t,x,v) &\ge \mathbf{1}_{\{\tau_0 \in [\Lambda,2\Lambda]\}} \underline{R}(y_0,v) c_{\frac{d(\Omega)}{\Lambda}, \frac{d(\Omega)}{\Lambda}}^{P-1} c_{\Lambda^2, \frac{d(\Omega)}{\Lambda}} \int_{D_\Lambda} f_0(z,v_P) \\
&\quad \times \int_\Lambda^{2\Lambda} \int_{U_{y_0}} \frac{|(y_1 - y_0) \cdot n_{y_0}| |(y_1 - y_0) \cdot n_{y_1}|}{\tau_1^5}   \\
&\quad \times \int_\Lambda^{2\Lambda} \int_{U_{y_1}} \frac{|(y_2 - y_1) \cdot n_{y_1}||(y_2 - y_1) \cdot n_{y_2}|}{\tau_2^5}  \\
&\quad \times \dots  \\
&\quad \times \int_\Lambda^{2\Lambda} \int_{U_{y_{P-2}}} \frac{|(y_{P-2} - y_{P-1})\cdot n_{y_{P-1}}| |(y_{P-2} - y_{P-1}) \cdot n_{y_{P-2}}|}{\tau_{P-1}^5}  \\
& \quad \times \frac{|(y_{P-1} - q(z,v_P)) \cdot n_{q(z,v_P)}| |(y_{P-1} - q(z,v_P)) \cdot n_{y_{P-1}}|}{\big(t - \sum_{i=0}^{P-1} \tau_i - \sigma(z,v_P)\big)^5} \, \\
&\quad \times \, \mathbf{1}_{\{q(z,v_P) \leftrightarrow y_{P-1} \}} \, \d \zeta(y_{P-1}) \, \d \tau_{P-1} \dots \d \zeta(y_1) \d \tau_1 \d v_P \d z.
\end{align*}
Since $\int_\Lambda^{2\Lambda} \frac{d\tau}{\tau^5} < \infty$, we deduce immediatly that for some constant $c_0$ independent of $(y_0,v)$, whose value may vary from line to line 
\begin{align}
\label{eq:DH_Final_3}
f(t,x,v) &\ge c_0  \mathbf{1}_{\{\tau_0 \in [\Lambda,2\Lambda]\}} \underline{R}(y_0,v) \int_{D_\Lambda} f_0(z,v_P) \\
&\quad \times \int_{U_{y_0}} |(y_1 - y_0) \cdot n_{y_0}| |(y_1 - y_0) \cdot n_{y_1}|   \nonumber \\
&\quad \times \int_{U_{y_1}} |(y_2 - y_1) \cdot n_{y_1}||(y_2 - y_1) \cdot n_{y_2}| \nonumber \\
&\quad \times \dots \nonumber \\
&\quad \times \int_{U_{y_{P-2}}} |(y_{P-2} - y_{P-1})\cdot n_{y_{P-1}}| |(y_{P-2} - y_{P-1}) \cdot n_{y_{P-2}}|  \nonumber \\
& \quad \times |(y_{P-1} - q(z,v_P)) \cdot n_{q(z,v_P)}| |(y_{P-1} - q(z,v_P)) \cdot n_{y_{P-1}}| \nonumber \\
&\quad \times \, \mathbf{1}_{\{q(z,v_P) \leftrightarrow y_{P-1} \}} \, \d \zeta(y_{P-1}) \dots \d \zeta(y_1)  \d v_P \d z. \nonumber
\end{align}

\vspace{.3cm}

\textbf{Step 4.} For a couple of points $(a,b) \in (\pO)^2$, we set 
\begin{align*}
h_P(a,b) &= \int_{U_a} |(y_1 - a) \cdot n_a| |(y_1 - a) \cdot n_{y_1}| \\
&\quad \times \int_{U_{y_1}} |(y_2 - y_1) \cdot n_{y_1}| |(y_2 - y_1) \cdot n_{y_2}| \times \dots \\
&\quad \times \int_{U_{y_{P-2}}} |(y_{P-1} - y_{P-2}) \cdot n_{y_{P-1}}| |(y_{P-1} - y_{P-2}) \cdot n_{y_{P-2}}| \\
&\quad \times |(y_{P-1} - b) \cdot n_b| |(y_{P-1} - b) \cdot n_{y_{P-1}}| \, \mathbf{1}_{\{b \leftrightarrow y_{P-1}\}} \, \d \zeta(y_{P-1}) \dots \d \zeta(y_1).
\end{align*}
In this step, we want to show that, for all $y_0 \in \pO$, $b \to h_P(y_0,b)$ is lower semicontinuous and positive. To this aim, we present a simplified proof of the argument given in \cite{Bernou_Transport_Semigroup_2020}. We can rewrite $h_P$ as 
\[ h_P(a,b) = \int_{\{(y_1, \dots,y_{P-1}) \in D(a,b)\}} N(a,y_1,\dots,y_{P-1}, b) \d \zeta(y_{P-1}) \dots \d \zeta(y_1), \]
with 
\begin{align*}
D(a,b) &:= \Big\{(y_1,\dots,y_{P-1}) \in (\pO)^{P-1}: \\
&\qquad \qquad y_1 \leftrightarrow a, y_2 \leftrightarrow y_1, \dots, y_{P-1} \leftrightarrow y_{P-2}, b \leftrightarrow y_{P-1} \Big\},
\end{align*}
and 
\begin{align*}
N(a,y_1,\dots,y_{P-1},b) &:= |(y_1 -a) \cdot n_{a}| |(y_1 - a) \cdot n_{y_1}| |(y_{P-1} - b) \cdot n_b| |(y_{P-1} - b) \cdot n_{y_{P-1}}| \\
&\quad \times \Pi_{i = 1}^{P-2} |(y_{i+1} - y_i) \cdot n_{y_i}| |(y_{i+1} - y_i) \cdot n_{y_{i+1}}|.
\end{align*}
By regularity assumption, if $(z_1,z_2) \in (\pO)^2$ with $z_1 \leftrightarrow z_2$, there exists $\epsilon > 0$ such that $B(z_1,\epsilon) \cap \pO \leftrightarrow B(z_2,\epsilon) \cap \pO$, i.e. for all $p \in B(z_1,\epsilon) \cap \pO$, $q \in B(z_2,\epsilon) \cap \pO$, we have $p \leftrightarrow q$, see \cite[Lemma 38]{bernou_fournier_collisionless_2019}. Combining this with the statement of Proposition \ref{prop:evans}, we find that 
\begin{align}
\label{eq:conclusion_Hausdorff}
 \mathcal{H}(D(a,b)) > 0, 
 \end{align}
where we recall that $\mathcal{H}$ denotes the $d-1$ dimensional Hausdorff measure. We set, for all $a \in \pO$, 
\[ D(a) := \Big \{(y_1,\dots,y_{P-1}) \in (\pO)^{P-1}: y_1 \leftrightarrow a, y_2 \leftrightarrow y_1, \dots, y_{P-1} \leftrightarrow y_{P-2} \Big\}. \]
For $a \in \pO$ and $(y_1,\dots,y_{P-1}) \in D(a)$, for all $b \in \pO$ such that $b \leftrightarrow y_{P-1}$, we have $N(a,y_1,\dots,y_{P-1},b) > 0$ according to Definition \ref{defi:communique}. Using \eqref{eq:conclusion_Hausdorff}, one concludes that for all $(a,b) \in (\pO)^2$, $h_P(a,b) > 0$. Moreover, the map $b \to N(a,y_1,\dots,y_{P-1},b)$ is continuous since $z \to n_z$ is continuous. According to \cite[Lemma 2.3]{Evans_1999}, for all $z \in \pO$, $U_z$ is open and non-empty. Hence for all $y_{P-1} \in \pO$, $b \to \mathbf{1}_{U_{y_{P-1}}}(b)$ is lower semicontinuous. We conclude that, for all $a \in \pO$, $(y_1,\dots,y_{P-1}) \in D(a)$,
\[ b \to N(a,y_1,\dots,y_{P-1},b) \mathbf{1}_{\{y_{P-1} \leftrightarrow b\}} \]
is lower semicontinuous. For $a \in \pO$, $(b_n)_{n \ge 1}$ a sequence of $\pO$ converging towards $b \in \pO$, we obtain 
\begin{align*}
0 < h_P(a,b) &\le \int_{D(a)} \liminf_{n \to \infty} N(a,y_1,\dots,y_{P-1},b_n) \mathbf{1}_{\{y_{P-1} \leftrightarrow b_n\}} \, \d \zeta(y_1) \dots \d \zeta(y_{P-1}) \\
&\le \liminf_{n \to \infty} h_P(a,b_n),
\end{align*}
using Fatou's lemma. Thus $\pO \ni b \to h_P(a,b)$ is also lower semicontinuous and positive for all $a \in \pO$. 

\vspace{.3cm}

\textbf{Step 5.} We use Step 4 to conclude the proof. Since $\pO$ is compact, we deduce from the previous step that for all $a \in \pO$, 
\[ \mu(a) := \inf_{b \in \pO} h_P(a,b) > 0. \]
With this at hand, we have from \eqref{eq:DH_Final_3}
\begin{align*}
f(t,x,v) &\ge c_0 \mathbf{1}_{\{\tau_0 \in [\Lambda,2\Lambda]\}} \underline{R}(y_0,v) \int_{D_\Lambda} f_0(z,v_P) \, h_P(y_0, q(z,v_P)) \, \d v_P \d z \\
&\ge c_0 \mathbf{1}_{\{\tau_0 \in [\Lambda,2\Lambda] \}} \underline{R}(y_0,v) \mu(y_0) \int_{D_\Lambda} f_0(z,v_P) \, \d v_P \d z
\end{align*}
and, recalling that $\tau_0 = \sigma(x,-v)$, $y_0 = q(x,-v)$, we set 
\[ \nu(x,v) = c_0 \mathbf{1}_{\{\sigma(x,-v) \in [\Lambda,2\Lambda]\}} \underline{R}(q(x,-v),v) \mu(q(x,-v)) \]
and $T(\Lambda):= t = (2P + 2)\Lambda$, which is indeed of the form $\kappa \Lambda$ for $\kappa = (2P+2) > 0$. 

Finally, note that if $f \in L^1(G)$ with $f \ge 0$, $\supp(f) \subset D_{\Lambda}$ and $\langle f \rangle = 1$, we have
\begin{align*}
S_T f(x,v) \ge \nu(x,v) \langle f \rangle,
\end{align*}
and integrating this equality over $G$ and using the mass conservation leads to 
\begin{align*}
\langle \nu \rangle \le 1. 
\end{align*}
\end{proof}

\begin{rmk}[Regarding the constructive property of $\nu$]
One might wonder whether the measure $\nu$ is explicit, leading to a constructive rate of convergence. There are two compactness arguments in the previous proof: one gives the value of the constant $c_0$ in Step 3, and is quite artificial. Indeed, for a given $\Lambda$ and fixed parameters of the boundary condition, one could easily find a constructive lower bound for the $c_{\cdot,\cdot}$ involved in the proof. The situation is a bit less clear for the compactness argument of Step 4, which is the same as the one used in the proof of the Doeblin-Harris condition for the Maxwell boundary condition. On this matter, we refer to \cite[Remark 8]{Bernou_Transport_Semigroup_2020}, where it is proven that a constructive lower bound can be found at least when $\Omega$ is the unit disk. More generally, we expect to be able to find a lower bound for every given $\Omega$.
\end{rmk}

\section{Proof of the main results}
\label{section:main}

As mentioned above, starting from the Lyapunov inequalities of Proposition \ref{prop:ineq_lyapunov} and the Doeblin-Harris condition, Theorem \ref{thm:Doeblin-Harris}, the proof of Theorem \ref{thm:Main} follows from the same strategy as the one applied in \cite{Bernou_Transport_Semigroup_2020} and introduced in \cite{Canizo_Mischler_2020}. We provide a full proof for the sake of completeness and to clarify all the required adaptations. Let us emphasize the fact that the inclusion of $|v|$ in the quantity $\langle x,v \rangle$, required to obtain the Lyapunov inequalities of Proposition \ref{prop:ineq_lyapunov}, prevents us from using the logarithm to derive the optimal rate of convergence as was done in \cite{Bernou_Transport_Semigroup_2020}. Instead, we can only use polynomial weights depending on some arbitrary small exponent $\epsilon$. We write $\vertiii{T}_{A \to B}$ for the operator norm of the linear operator $T$ acting from $A$ to $B$.

\subsection{Contraction property in well-chosen norm}

The following lemma introduces new norms based on the weights for which Lyapunov inequalities were established in Section \ref{section:Lyapunov}. We obtain a norm in which the semigroup $(S_t)_{t \ge 0}$ is more than a contraction in the large sense. Recall the definition of $\langle\cdot, \cdot \rangle$ from \eqref{eq:def_bracket}.

\begin{lemma}[Contraction in well-chosen norm]
\label{lemma:contraction_norm}
Fix $\epsilon \in (0,1)$ and, for $p \in (1+\epsilon,d+1]$, set $m^{\epsilon}_{p}(x,v) = \langle x,v \rangle^{p-\epsilon}$ on $G$. There exists $T_0 > 0$ such that for all $T \ge T_0$, there exist $\beta(T) > 0$, $\alpha(T) = C_3 \beta(T) T$ with $C_3 > 0$ constant such that, for all $f \in L^1_{m^{\epsilon}_{d+1}}(G)$ with $\langle f \rangle = 0$, we have
\begin{align}
\label{eq:Contraction_Lemma}
\|S_T f\|_{L^1} + \beta \|S_T f\|_{m^{\epsilon}_{d+1}} + \alpha \|S_T f\|_{m^{\epsilon}_d} \le \|f\|_{L^1} + \beta \|f\|_{m^{\epsilon}_{d+1}} + \frac{\alpha}{3} \|f\|_{m^{\epsilon}_{d}}, 
\end{align}
so that, setting
\[ \vertiii{\cdot}_{m^{\epsilon}_{d+1}} := \|\cdot\|_{L^1} + \beta \|\cdot\|_{m^{\epsilon}_{d+1}} + \alpha \|\cdot\|_{m^{\epsilon}_d}, \]
there holds $\vertiii{S_T f}_{m^{\epsilon}_{d+1}} \le \vertiii{f}_{m^{\epsilon}_{d+1}}$. Moreover, there exists $M_{d+1}^{\epsilon} > 1$ such that for all $f \in L^1_{m^{\epsilon}_{d+1}}(G)$ with $\langle f \rangle = 0$, 
\[ \| S_T f\|_{m^{\epsilon}_{d+1}} \le M_{d+1}^{\epsilon} \|f\|_{m^{\epsilon}_{d+1}}. \]
\end{lemma}

\begin{proof}
\textbf{Step 1.} We use Proposition \ref{prop:ineq_lyapunov} to obtain a new integral inequality. For all $T > 0$, according to the lemma, there exists $C_1,C_2,b_1,b_2 > 0$ such that for all $f \in L^1_{m^{\epsilon}_{d+1}}(G)$,

\vspace{-.5cm}
\begin{subequations} 
\begin{align}
\| S_T f\|_{m^{\epsilon}_{d+1}} + C_1 \int_0^T \|S_t f\|_{m^{\epsilon}_{d}} \, \d t  &\le \|f\|_{m^{\epsilon}_{d+1}} + b_1 (1 + T) \|f\|_{L^1},  \label{eq:Step_1_1} \\
\hbox{ and } \| S_T f\|_{m^{\epsilon}_{d}} + C_2 \int_0^T \|S_t f\|_{m^{\epsilon}_{d-1}} \, \d t  &\le \|f\|_{m^{\epsilon}_{d}} + b_2 (1 + T) \|f\|_{L^1}. \label{eq:Step_1_2}
\end{align}
\end{subequations}
Let $t \in (0,T)$. We deduce first from \eqref{eq:Step_1_2},
\[ \| S_{T-t} S_t f\|_{m^{\epsilon}_{d}} \le \|S_t f\|_{m^{\epsilon}_d} + b_2 (1 + T - t) \|S_t f\|_{L^1}, \]
which we rewrite as 
\[ \|S_T f\|_{m^{\epsilon}_d} - b_2 (1 + T - t) \|S_t f\|_{L^1} \le \|S_t f\|_{m^{\epsilon}_d}. \]
We plug this inside \eqref{eq:Step_1_1} to obtain:
\begin{align*}
\|S_T f\|_{m^{\epsilon}_{d+1}} &+ C_1 \int_0^T \Big( \|S_T f\|_{m^{\epsilon}_d} - b_2 (1 + T - t) \|S_t f\|_{L^1} \Big) \d t \\ 
&\le \|f\|_{m^{\epsilon}_{d+1}} + b_1 (1 + T) \|f\|_{L^1}.
\end{align*}
Finally, we can use the $L^1$ contraction from Theorem \ref{thm:mass_&_L1_contraction} to get
\begin{align}
\label{eq:Step_1_Ccl}
\|S_T f\|_{m^{\epsilon}_{d+1}} + C_1 T \|S_T f\|_{m^{\epsilon}_d} \le \|f\|_{m^{\epsilon}_{d+1}} + b_1' ( 1 + T + T^2) \|f\|_{L^1},
\end{align}
with $b_1' > 0$ constant, independent of $T$. 

\vspace{.3cm}

\textbf{Step 2.} According to Theorem \ref{thm:Doeblin-Harris}, for all $\rho > 2$, there exists $T(\rho) = \xi \rho$ for some constant $\xi > 0$ and a measure $\nu$ on $G$ with $\nu \not \equiv 0$ such that 
\[ S_{T(\rho)} h \ge \nu \int_{\{(x,v) \in G, \langle x,v \rangle \le \rho\}} h \, \d v \d x, \]
for all $h \in L^1(G)$ with $h \ge 0$. 
By assumption, $f \in L^1_{m^{\epsilon}_{d+1}}(G)$ and $\langle f \rangle = 0$. We set, for any $\rho > 0$, $\bar{m}^{\epsilon}_d(\rho) = \rho^{d-\epsilon}$, and $\kappa(\rho) = \frac{b_1'(1 + T + T^2)}{T}(\rho)$. Since $T(\rho) = \xi \rho$ for some constant $\xi > 0$, $\kappa(\rho) \underset{\rho \to \infty}{\sim} C \rho$ for some $C > 0$. Since $d \in \{2,3\}$ and $\epsilon \in (0,1)$, one can find $\rho_0$ such that, for all $\rho > \rho_0$, $\bar{m}^{\epsilon}_d (\rho) \ge \frac{12 \kappa(\rho)}{C_1}$. We fix $\rho > \rho_0$, $T = T(\rho) > T(\rho_0) =: T_0$ for the remaining part of the proof. Note that, since $T(\rho) = \xi \rho$ for some constant $\xi$, any choice of $T > T(\rho_0)$ is possible. We set $A := \frac{\bar{m}^{\epsilon}_d(\rho)}{4}$ and define, for all $\beta > 0$, the $\beta$-norm by
\[ \|f\|_{\beta} := \|f\|_{L^1} + \beta \|f\|_{m^{\epsilon}_{d+1}}.  \]
We distinguish two cases. Indeed, we have the alternative: 
\begin{subequations}
\begin{align}
\|f\|_{m^{\epsilon}_d} &\le A \|f\|_{L^1}, \label{eq:Step_2_1} \\
\hbox{ or } \|f\|_{m^{\epsilon}_d} &> A \|f\|_{L^1}.  \label{eq:Step_2_2}
\end{align}
\end{subequations}

\vspace{.3cm}

\textbf{Step 3.} We prove a convergence result in the $\beta$-norm in the case of the first alternative, \eqref{eq:Step_2_1}. Recall that for all $\Lambda > 0$, $D_\Lambda = \{(x,v) \in G, \langle x,v \rangle \le \Lambda\}$. Using $\langle f \rangle = 0$ and Theorem \ref{thm:Doeblin-Harris}, we have, for all $(x,v) \in G$, 
\[ \begin{aligned}
S_Tf_{\pm}(x,v) &\geq \nu(x,v) \int_{G} f_{\pm}(x',v') \,  \d v' \d x' - \nu(x,v) \int_{D_{\rho}^c} f_{\pm}(x',v') \, \d v' \d x' \\
& \geq \frac{\nu(x,v)}{2} \int_G |f(x',v')| \, \d v' \d x'  - \nu(x,v) \int_{D_{\rho}^c} |f(x',v')|\, \d v' \d x' \\
& \geq \frac{\nu(x,v)}{2} \int_G |f(x',v')| \, \d v' \d x'  - \frac{\nu(x,v)}{\bar{m}^{\epsilon}_d(\rho)} \int_{G}  |f(x',v')| m^{\epsilon}_{d}(x',v') \, \d v' \d x' \\
& \geq \frac{\nu(x,v)}{2} \int_G |f(x',v')| \, \d v' \d x'  - \frac{\nu(x,v)}{4}  \int_{G}  |f(x',v')| \, \d v' \d x' \\
&= \frac{\nu(x,v)}{4}  \int_{G}  |f(x',v')| \, \d v' \d x' =: \eta(x,v),
\end{aligned} \]
where the third inequality is given by the fact that $D_{\rho} = \{(x,v) \in G, m^{\epsilon}_d(x,v) \le \bar{m}^{\epsilon}_d(\rho)\}$ and that $m^{\epsilon}_d(x,v) \ge 1$ for all $(x,v) \in G$. The last inequality is obtained by condition (\ref{eq:Step_2_1}). The final equality stands for a definition of $\eta(x,v)$ for all $(x,v) \in G$. Note that $\eta \geq 0$ on $G$.
We deduce,
\[ \begin{aligned}
|S_T f| &= |S_T f_+ - \eta - (S_T f_- - \eta)| \\
		&\leq |S_T f_+ - \eta| + |S_T f_- - \eta| \\
		&= S_T f_+ + S_T f_- - 2 \eta = S_T|f| - 2 \eta,
\end{aligned} \]
and, integrating over $G$, we have, using the contraction property, that $\eta = \frac{\nu}{4} \|f\|_{L^1}$, and that $\nu$ is non-negative with $\langle \nu \rangle \le 1$,
\begin{equation}\begin{aligned}
\label{IneqContractionDoeblin}
\|S_T f \|_{L^1} \leq \|f\|_{L^1} - 2 \|\eta\|_{L^1} = \Big( 1- \frac{\langle \nu \rangle}{2} \Big) \|f\|_{L^1} = \tilde{\eta} \|f\|_{L^1},
\end{aligned}\end{equation}
with $\tilde{\eta} \in (0, 1)$. Hence, $S_T$ is a strict contraction in $L^1$ in the case where $f$ satisfies \eqref{eq:Step_2_1}.
We use this result along with \eqref{eq:Step_1_Ccl} and the definition of $\kappa(\rho)$ to derive an inequality on the $\beta$-norm of $S_T f$
\[ \begin{aligned}
\|S_T f\|_{\beta} &= \|S_T f\|_{L^1} + \beta \|S_T f\|_{m^{\epsilon}_{d+1}} \\
&\leq \tilde{\eta} \|f\|_{L^1} + \beta \big( - C_1 T \|S_T f\|_{m^{\epsilon}_d} + \|f\|_{m^{\epsilon}_{d+1}} + \kappa(\rho) T \|f\|_{L^1} \big) \\
&\leq \beta \|f\|_{m^{\epsilon}_{d+1}} + (\tilde{\eta}+ \kappa(\rho) T\beta ) \|f\|_{L^1} - \beta C_1 T \|S_T f\|_{m^{\epsilon}_d}.
\end{aligned} \]
Finally, we choose $0 < \beta \leq \frac{1 - \tilde{\eta}}{ \kappa(\rho) T}$ and deduce
\begin{equation}\begin{aligned}
\label{EqPreliminaryBeta}
\|S_T f\|_{\beta} + C_1 \beta T \|S_T f\|_{m^{\epsilon}_d} \leq \|f\|_{\beta}.
\end{aligned}\end{equation}

\vspace{.3cm}

\textbf{Step 4.} We prove that a slightly different version of (\ref{EqPreliminaryBeta}) also holds in the case (\ref{eq:Step_2_2}).
From (\ref{eq:Step_1_Ccl}), using (\ref{eq:Step_2_2}), we have, for $T$, $\kappa(\rho)$ fixed as above
\[ \begin{aligned}
\|S_T f\|_{m^{\epsilon}_{d+1}} + C_1 T \|S_T f\|_{m^{\epsilon}_{d}} &\leq \|f\|_{m^{\epsilon}_{d+1}} + \frac{\kappa(\rho) T}{A} \|f\|_{m^{\epsilon}_d}.
\end{aligned} \]
Since $A \geq \frac{3 \kappa(\rho)}{C_1}$, see Step 2, it follows that
\[ \begin{aligned}
\|S_T f\|_{m^{\epsilon}_{d+1}} + C_1 T \|S_T f\|_{m^{\epsilon}_d} &\leq  \|f\|_{m^{\epsilon}_{d+1}} + \frac{C_1T}{3} \|f\|_{m^{\epsilon}_d}.
\end{aligned} \]
Using this inequality and the $L^1$ contraction, we deduce
\begin{align}
\label{EqPreliminaryBetaCase2}
\|S_T f\|_{\beta} + C_1 \beta T \|S_T f\|_{m^{\epsilon}_d} &= \|S_T f\|_{L^1} + \beta \|S_T f\|_{m^{\epsilon}_{d+1}} + C_1 \beta T \|S_T f\|_{m^{\epsilon}_d}  \nonumber \\
&\leq \|f\|_{L^1} + \beta \|f\|_{m^{\epsilon}_{d+1}} + \beta \frac{C_1 T}{3} \|f\|_{m^{\epsilon}_{d}}  \nonumber \\
&= \|f\|_{\beta} + \beta C_1 \frac{T}{3} \|f\|_{m^{\epsilon}_d}.
\end{align}

\vspace{.3cm}

\noindent
\textbf{Step 5.} For $\beta$ as above and $\alpha = C_1 \beta T$, we have $\vertiii{.}_{m^{\epsilon}_{d+1}} = \|.\|_{\beta} + \alpha \|.\|_{m^{\epsilon}_d}$ by definition. Gathering (\ref{EqPreliminaryBeta}) and (\ref{EqPreliminaryBetaCase2}), we conclude that (\ref{eq:Contraction_Lemma}) holds and we deduce
\[ \begin{aligned}
\vertiii{S_T f}_{m^{\epsilon}_{d+1}} \leq \vertiii{ f}_{m^{\epsilon}_{d+1}}.
\end{aligned} \]

Since $m^{\epsilon}_{d+1} \geq m^{\epsilon}_{d} \geq 1$ on $G$, we conclude that for all $f \in L^1_{m^{\epsilon}_{d+1}}(G)$ with $\langle f \rangle = 0$,
\begin{equation}\begin{aligned}
\label{IneqM3Contraction}
\|S_T f\|_{m^{\epsilon}_{d+1}} \leq M^{\epsilon}_{d+1} \|f\|_{m^{\epsilon}_{d+1}},
\end{aligned}\end{equation}
for some constant $M^{\epsilon}_{d+1} \geq 1$.
\end{proof}

To derive interpolation results between spaces of the form $\{f \in L^1_w(G), \langle f \rangle = 0\} $ with $w \ge 1$ some weight on $G$, we will rely on \cite[Corollary 3]{Bernou_Transport_Semigroup_2020}, that we recall now. 

\begin{coroll}
\label{coroll:Interpol}
Let $\phi_1, \phi_2, \tilde{\phi}_1, \tilde{\phi}_2$ be four measurable functions on $G$ positive almost everywhere. Let also $A_1 = L^1_{\phi_1}(G)$, $A_2 = L^1_{\phi_2}(G)$, $\tilde{A}_1 = L^1_{\tilde{\phi}_1}(G)$, $\tilde{A}_2 = L^1_{\tilde{\phi}_2}(G)$. 
Let, for all $\gamma \in (0,1)$, $\phi_{\gamma}$ and $\tilde{\phi}_{\gamma}$ be defined by
\[ \phi_{\gamma} := \phi_1^{\gamma} \phi_2^{1-\gamma}, \qquad \tilde{\phi}_{\gamma} := \tilde{\phi}_1^{\gamma} \tilde{\phi}_2^{1-\gamma}, \]
respectively, and $A_{\gamma} = L^1_{\phi_{\gamma}}(G)$, $\tilde{A}_{\gamma} = L^1_{\tilde{\phi_{\gamma}}}(G)$. Assume that there exists a bounded projection $\Pi: (A_i,\tilde{A}_i) \to (A_i',\tilde{A}_i')$ for $i \in \{1,2\}$ with $A_i' \subset A_i$, $\tilde{A}_i' \subset \tilde{A}_i$. Let also $A'_{\gamma} = (A'_1 + A'_2) \cap A_{\gamma}$, $\tilde{A}'_{\gamma} = (\tilde{A}'_1 + \tilde{A}'_2) \cap \tilde{A}_{\gamma}$. Assume that $S$ is a linear operator from $A'_1$ to $\tilde{A}'_1$ and from $A'_2$ to $\tilde{A}'_2$ with
\[ \vertiii{S}_{A'_1 \to \tilde{A}'_1} \leq N_1, \qquad \vertiii{S}_{A_2' \to \tilde{A}_2'} \leq N_2, \]
for $N_1, N_2 > 0$. Then $S$ is a linear operator from $A'_{\gamma}$ to $\tilde{A}'_{\gamma}$ and there exists $C > 0$ depending only on $\Pi$ such that
\[ \vertiii{S}_{A'_{\gamma} \to \tilde{A}'_{\gamma}} \leq C N_1^{\gamma} N_2^{1- \gamma}. \]
\end{coroll}

\subsection{Proof of Theorem \ref{thm:Main}}

In this subsection, we proceed to the proof of Theorem \ref{thm:Main}. 

For $\epsilon \in (0,\frac12)$ fixed, we consider the weights $w_1(x,v) = \langle x,v \rangle^{1+\epsilon}$ and $w_0(x,v) = \langle x, v \rangle^{\epsilon}$ for all $(x,v) \in \bar{G}$. We want to prove a decay rate for $S_t(f-g)$ with $f,g \in L^1_{m_{d+1}^{\epsilon}}$, $\langle f \rangle = \langle g \rangle$. We assume without loss of generality that $g \equiv 0$ so that $f \in L^1_{m^{\epsilon}_{d+1}}(G)$ with $\langle f \rangle = 0$.

\vspace{.3cm}

\textbf{Step 1.} Set $L^1_0(G) = \{g \in L^1(G), \langle g \rangle = 0\}$ and $L^1_{w,0}(G) = \{g \in L^1_w(G), \langle g \rangle = 0 \}$ for any weight $w$ on $\bar{G}$. We introduce the notation 
\[ M_1(v) := \frac{e^{-\frac{|v|^2}{2}}}{(2 \pi)^{\frac{d}{2}}}. \]
Note that $\int_{\RR^d} |v|^2 \, M_1(v) \, \d v = 1$.
We consider $\Pi: L^1(G) \to L^1_0(G)$ the bounded projection such that, for all $h \in L^1(G)$, $(x,v) \in G$, 
\[ \Pi h(x,v) = h(x,v) - \frac{M_1(v)|v|^2}{|\Omega|} \int_G h(y,w) \, \d y \d w, \]
where $|\Omega|$ denotes the volume of $\Omega$. By use of hyperspherical coordinates, it is straightforward to check that $\Pi h \in L^1_{m^{\epsilon}_{d+1}}(G)$ for all $h \in L^1_{m^{\epsilon}_{d+1}}(G)$. Also, there exists a constant $C_\Pi > 0$ such that $\|\Pi h \|_{m^{\epsilon}_{d+1}} \le C_\Pi \|h\|_{m^{\epsilon}_{d+1}}$ for all $h \in L^1_{m^{\epsilon}_{d+1}}(G)$ and $\|\Pi h \|_{L^1} \le C_\Pi \|h\|_{L^1}$. Since $\langle h \rangle = 0$ implies $\Pi h = h$, and $\langle \Pi h \rangle = 0$ for all $h \in L^1(G)$, $\Pi$ is a bounded projection as claimed. 
Let $T > T_0$ with $T_0$ given by Lemma \ref{lemma:contraction_norm}. From Theorem \ref{thm:mass_&_L1_contraction}, we have
\begin{align*}
\vertiii{S_T}_{L^1_0(G) \to L^1_0(G)} \le 1, 
\end{align*}
and from Lemma \ref{lemma:contraction_norm}, 
\[ \vertiii{S_T}_{L^1_{m^{\epsilon}_{d+1},0}(G) \to L^1_{m^{\epsilon}_{d+1},0}(G)} \le M^{\epsilon}_{d+1}. \]
We apply Corollary \ref{coroll:Interpol} with the projection $\Pi$ and the values:
\begin{enumerate}[1.]
\item $A_1 = \tilde{A}_1 = L^1(G)$, and, using the definition of $\Pi$, $A_1' = \tilde{A}_1' = L^1_0(G)$,
\item $A_2 = \tilde{A}_2 = L^1_{m^{\epsilon}_{d+1}}(G)$, and, using the definition of $\Pi$, $A'_2 = \tilde{A}_2' = L^1_{m^{\epsilon}_{d+1},0}(G)$,
\item $\gamma = 1 - \frac{\epsilon}{d + 1 - \epsilon} \in (0,1)$, so that $A_{\gamma} = \tilde{A}_{\gamma} = L^1_{w_0}(G)$, and, using the definition of $\Pi$, $\tilde{A}_{\gamma}' = A_{\gamma}' = (A_1' + A_2') \cap A_{\gamma} = L^1_{w_0,0}(G)$. 
\end{enumerate}
We conclude that there exists $C_0 > 0$ such that
\[ \|S_T f\|_{w_0} \le C_0 \|f\|_{w_0}. \]
Since $(S_t)_{t \ge 0}$ is a strongly continuous semigroup of operators on $L^1_{w_0}(G)$, this implies, using the growth bound of the semigroup, that there exists $W_0 \ge 1$ such that for all $t \in (0,T)$, for all $f \in L^1_{w_0,0}(G)$, 
\begin{align}
\label{eq:control_St_w0}
\|S_T f\|_{w_0} = \|S_{T-t} S_t f \|_{w_0} \le W_0 \|S_t f\|_{w_0}. 
\end{align}

\vspace{.3cm}
\textbf{Step 2.} Using Proposition \ref{prop:ineq_lyapunov} and \eqref{eq:control_St_w0}, for some constants $C, W_1 > 0$, we have
\[ \|S_T f\|_{w_1} + \frac{T}{W_1} \|S_T f\|_{w_0} \le \|f\|_{w_1} + C(1+T) \|f\|_{L^1}, \]
which rewrites 
\[ \|S_T f\|_{w_1} + \frac{T}{W_1} \|S_T f\|_{w_0} \le \|f\|_{w_1} + \kappa(\rho) T \|f\|_{L^1}, \]
with, for all $\rho > 0$, $\kappa(\rho) = \frac{C(1 + T(\rho))}{T(\rho)}$, so that $\kappa \le C_{1,1}$ for some constant $C_{1,1} > 0$ independent of $\rho$. Set $w_0(r) = r^{\epsilon}$ for $r \ge 1$. Since $\frac{w_0(\rho)}{\kappa(\rho)} \to \infty$ when $\rho \to \infty$, one can replicate the arguments of Steps 2 to 4 of the proof of Lemma \ref{lemma:contraction_norm}. We obtain, for some $\tilde{T}_0 > T_0$, for all $T \ge \tilde{T}_0$,
\begin{align}
\label{eq:control_St_beta_w1}
\|S_T f\|_{\beta} + 3 \alpha \|S_T f\|_{w_0} \le \|f\|_{\beta} + \alpha \|f\|_{w_0},
\end{align}
just as \eqref{EqPreliminaryBeta} and \eqref{EqPreliminaryBetaCase2}, with $\beta > 0$ constant, $\alpha = \frac{\beta T}{3 W_1}$, and 
\begin{align}
\label{eq:def_beta_norm_w1}
\|f\|_{\beta} := \|f\|_{L^1} + \beta \|f\|_{w_1}. 
\end{align}

\vspace{.3cm}
\textbf{Step 3.} We have, from our definition of $w_0, w_1$ and of $m_{d+1}^{\epsilon}$, for $(x,v) \in G$, 
\begin{align*}
w_1(x,v) &= \langle x,v \rangle^{1+\epsilon} \\
&= \langle x,v \rangle^{1+\epsilon} \mathbf{1}_{\{\langle x,v \rangle < \lambda \}} + \langle x,v \rangle^{1+\epsilon} \mathbf{1}_{\{\langle x,v \rangle \ge \lambda \}} \\
& \le w_0(x,v) \lambda + \frac{\langle x,v \rangle^{d+1-\epsilon}}{\langle x,v \rangle^{d - 2 \epsilon}} \mathbf{1}_{\{ \langle x,v \rangle \ge \lambda \}} \\
&\le w_0(x,v) \lambda + \eta_{\lambda} m_{d+1}^{\epsilon},
\end{align*}
for $\lambda > 0$ large enough, with $\eta_{\lambda} = \frac{1}{\lambda^{d-2\epsilon}} \to 0$ as $\lambda \to \infty$. We deduce, since $w_1(x,v) \ge 1$ for all $(x,v) \in G$, 
\begin{align}
\label{eq:control_beta_lambda}
\frac{1}{\lambda(1+\beta)} \|S_T f\|_{\beta}  = \frac{1}{\lambda(1+\beta)} \big( \|S_T f\|_{L^1} + \beta \|S_T f\|_{w_1} \big) &\le \frac{1}{\lambda} \|S_T f\|_{w_1} \\
&\le \|S_T f\|_{w_0} + \frac{\eta_{\lambda}}{\lambda} \|S_T f\|_{m_{d+1}^{\epsilon}}.  \nonumber
\end{align}

Moreover, consider the norm $\vertiii{.}_{m^{\epsilon}_{d+1}}$ from Lemma \ref{lemma:contraction_norm}, and denote $\tilde{\beta}, \tilde{\alpha}$ the two positive constants used to define this norm. Setting $B := \frac{\alpha}{\tilde{\beta}}$, we have
\begin{align}
\label{eq:control_eta_lambda_vertiii}
\frac{\alpha \eta_{\lambda}}{\lambda}\|S_T f\|_{m^{\epsilon}_{d+1}} = \frac{\alpha}{\tilde{\beta}} \frac{\eta_{\lambda}}{\lambda} \tilde{\beta} \|S_T f\|_{m^{\epsilon}_{d+1}} \le B \frac{\eta_{\lambda}}{\lambda} \vertiii{S_T f}_{m^{\epsilon}_{d+1}},
\end{align}
with the definition given in Lemma \ref{lemma:contraction_norm} for $\vertiii{\cdot}_{m^{\epsilon}_{d+1}}$. Let $\gamma := \frac{\alpha}{1+\beta}$, $Z := 1 +\frac{\gamma}{\lambda}$, with $\lambda \ge \lambda_0 \ge 1$ and $\lambda_0$ large enough so that $Z \le 2$. We have
\begin{align*}
Z \big( \|S_T f\|_{\beta} + \alpha \|S_T f\|_{w_0} \big) &\le \|S_T f\|_{\beta} + \frac{\alpha}{\lambda(1+\beta)} \|S_T f\|_{\beta} + Z \alpha \|S_T f\|_{w_0} \\
&\le \|S_T f\|_{\beta} + \alpha \|S_T f\|_{w_0} + \frac{\alpha \eta_{\lambda}}{\lambda} \|S_T f\|_{m^{\epsilon}_{d+1}} + Z \alpha \|S_T f\|_{w_0} \\
&\le \|S_T f\|_{\beta} + 3 \alpha \|S_T f\|_{w_0} + \frac{B \eta_{\lambda}}{\lambda} \vertiii{S_T f}_{m^{\epsilon}_{d+1}} \\
&\le \|f\|_{\beta} + \alpha \|f\|_{w_0} + \frac{B \eta_{\lambda}}{\lambda} \vertiii{S_T f}_{m^{\epsilon}_{d+1}},
\end{align*}
where we used \eqref{eq:control_beta_lambda}, \eqref{eq:control_eta_lambda_vertiii} and  \eqref{eq:control_St_beta_w1}, in that order. We now introduce the norm $\vertiii{\cdot}_{w_1}$ defined, for all $h \in L^1_{w_1}(G)$, by
\[ \vertiii{h}_{w_1} := \|h\|_{\beta} + \alpha \|h\|_{w_0}, \]
so that the previous inequality rewrites 
\begin{align}
\label{eq:ineq_Z}
 Z \vertiii{S_T f}_{w_1} \le \vertiii{f}_{w_1} + \frac{B \eta_{\lambda}}{\lambda} \vertiii{S_T f}_{m^{\epsilon}_{d+1}}. 
 \end{align}

\vspace{.3cm}

\textbf{Step 4.} We set $u_0 = \vertiii{f}_{w_1}$, and, for $k \ge 1$, $u_k = \vertiii{S_{kT} f}_{w_1}$. Let $v_0 = \vertiii{f}_{m^{\epsilon}_{d+1}}$, and, for $k \ge 1$, $v_k = \vertiii{S_{kT} f}_{m^{\epsilon}_{d+1}}$. According to Lemma \ref{lemma:contraction_norm}, $v_k \le v_0$ for all $k \ge 1$. We set $Y = \frac{B \eta_{\lambda}}{\lambda}$. The final inequality \eqref{eq:ineq_Z} of Step 3 rewrites
\[ Z u_1 \le u_0 + Y v_1. \]
We iterate this inequality to obtain
\begin{align*}
Z^k u_k \le u_0 + Y \sum_{i=1}^k Z^{i-1} v_i,
\end{align*}
and we conclude that
\[ u_k \le Z^{-k}u_0 + Y \frac{Z}{Z-1} \sup_{i \ge 1} v_i \le Z^{-k} u_0 + \frac{YZ}{Z-1} v_0. \]
From this we deduce, recalling the definition of the $\beta$-norm \eqref{eq:def_beta_norm_w1}, that $1 < Z \le 2$ and that $w_1 \le m^{\epsilon}_{d+1}$, 
\begin{align*}
\vertiii{S_{kT} f}_{w_1} &\le \frac{1}{(1+ \gamma/\lambda)^k} (1 + \beta + \alpha) \|f\|_{w_1} + \eta_{\lambda} \frac{2B}{\gamma} \vertiii{f}_{m^{\epsilon}_{d+1}} \\
&\le C_{1,2} \Big( e^{-\frac{kT}{\lambda} \frac{\gamma}{2T}} + \eta_{\lambda} \Big) \|f\|_{m^{\epsilon}_{d+1}},
\end{align*}
with $C_{1,2} > 0$ explicit, where we used that $\vertiii{\cdot}_{m^{\epsilon}_{d+1}} \lesssim \|\cdot\|_{m^{\epsilon}_{d+1}}$. We set $T_1 = kT$ and choose 
\[ \lambda = \Big( \frac{T_1 \frac{\gamma}{2T}}{\ln(T_1^d)} \Big), \]
with $k \ge k_0 \ge 1$, $k_0$ large enough so that $\lambda > \lambda_0$ and $T_1 > e^1$ to obtain
\begin{align*}
\vertiii{S_T f}_{w_1} &\le C_{1,3}(d) \Big( \frac{1}{T_1^d} + \frac{\ln(T_1)^{d - 2\epsilon}}{T_1^{d-2\epsilon}} \Big) \|f\|_{m^{\epsilon}_{d+1}} \\
&\le \frac{ C_{1,3}(d)}{T_1^{d-3\epsilon}}\|f\|_{m^{\epsilon}_{d+1}}
\end{align*}
where $C_{1,3}(d) > 0$ is a constant depending only on $d$, and where we used that $\ln(T_1)^{d-2\epsilon} \le T_1^{\epsilon}$ for $k_0$ large enough. Upon modifying the value of $C_{1,3}(d)$ so that the previous inequality also holds for $k \in \{0,\dots,k_0 - 1\}$, we can rewrite this as
\begin{align}
\label{eq:conclusion_Step_4}
\vertiii{S_{kT} f}_{w_1} \le C_{1,3}(d) \Theta(k) \|f\|_{m^{\epsilon}_{d+1}},
\end{align}
with $\Theta(k) = \frac{1}{(kT)^{d-3\epsilon}}$ for all $k \ge 1$.

\vspace{.3cm}

\textbf{Step 5.} With the norm $\vertiii{\cdot}_{w_1}$, \eqref{eq:control_St_beta_w1} rewrites
\[ \vertiii{S_T f}_{w_1} + 2 \alpha \|S_T f\|_{w_0} \le \vertiii{f}_{w_1}. \]
By iterating this inequality and summing, we obtain, for $l \ge 1$, writing $[x]$ for the floor of $x \in \RR$, 
\begin{align}
\label{eq:intermediate_Step_5}
0 \le \vertiii{S_{lT} f}_{w_1} + 2 \alpha \sum_{k = [ \frac{l}{2} ] + 1}^l \|S_{kT} f\|_{w_0} \le \vertiii{S_{[\frac{l}{2}] T} f }_{w_1}.
\end{align}

Note that, for any $1 \le k \le l$, 
\[ \|S_{lT} f\|_{L^1} \le \|S_{kT} f\|_{L^1} \le \|S_{kT} f\|_{w_0}. \]
Hence, using \eqref{eq:conclusion_Step_4} and \eqref{eq:intermediate_Step_5}, 
\[ \min(1,2\alpha) \Big( l - [\frac{l}{2}] + 1\Big) \|S_{lT} f\|_{L^1} \le C_{1,3}(d) \Theta \Big([\frac{l}{2}] \Big) \|f\|_{m^{\epsilon}_{d+1}}, \]
so that, allowing the value of $C_{1,3}(d)$ to change slightly,
\[ \|S_{lT} f\|_{L^1} \le C_{1,3}(d) \frac{1}{(lT)^{d+1 - 3\epsilon}} \|f\|_{m^{\epsilon}_{d+1}}. \]
We conclude to the desired rate by choosing $\epsilon' = 3 \epsilon > 0$ and using standard semigroup properties.

\subsection{Proof of Theorem \ref{thm:equilibrium}} We use the previous Theorem \ref{thm:Main} to prove Theorem \ref{thm:equilibrium}. We obtain a first result by interpolation:
\begin{lemma}
\label{lemma:interpolation_final}
For all $\epsilon \in (0,\frac14)$, there exists an explicit constant $C_{\epsilon}$ such that for all $t \ge 0$, for all $f,g \in L^1_{m^{\epsilon}_{d}}(G)$ with $\langle f \rangle = \langle g \rangle$, there holds
\[ \|S_t(f-g)\|_{L^1} \le C_{\epsilon} \frac{1}{(1+t)^{d-\epsilon}} \|f-g\|_{m^{\epsilon}_{d}}. \]
\end{lemma}

\begin{proof}
We set $\tilde{f} := f - g$ so that $\langle \tilde{f} \rangle = 0$ and $\tilde{f} \in L^1_{m^{\epsilon}_{d},0}(G)$. From Theorem \ref{thm:mass_&_L1_contraction}, we have, for all $t \ge 0$,
\[ \vertiii{S_t}_{L^1_0(G) \to L^1_0(G)} \le 1, \]
and from Theorem \ref{thm:Main}, for all $\epsilon_0 \in (0,\frac14)$,
\[ \vertiii{S_t}_{L^1_{m^{\epsilon_0}_{d+1},0}(G) \to L^1_{0}(G)} \le C \frac{1}{t^{d+1-\epsilon_0}} = C \tilde{\Theta}(t), \]
the last equality standing for a definition of $\tilde{\Theta}(t)$, with $C > 0$ independent of $t$. We introduce, as before, the projection $\Pi : L^1(G) \to L^1_0(G)$, given, for $h \in L^1(G)$, by
\[ \Pi h(x,v) = h(x,v) - \frac{M_1(v)|v|^2}{|\Omega|} \int_G h(y,w) \, \d w \d y, \quad (x,v) \in G. \]
Note that, if $h \in L^1_{m^{\epsilon_0}_{d+1}}(G)$, $\Pi h \in L^1_{m^{\epsilon_0}_{d+1},0}(G)$ as one can check using hyperspherical coordinates, and that $\langle \Pi h \rangle = 0$. Moreover, $\Pi $ sends $L^1_r(G)$ to $L^1_{r,0}(G)$ for any weight $1 \le r \le m^{\epsilon_0}_{d+1}$, and is bounded. We apply Corollary \ref{coroll:Interpol} with the projection $\Pi $ and
\begin{enumerate}
\item $A_1 = \tilde{A}_1 = \tilde{A}_2 = L^1(G)$, 
\item $A_2 = L^1_{m^{\epsilon_0}_{d+1}}(G)$,
\item $A_1' = \tilde{A}_1' = \tilde{A}_2' = L^1_0(G)$, $A_2' = L^1_{m^{\epsilon_0}_{d+1},0}(G)$,
\item $\gamma = 1 - \frac{d - \epsilon}{d+1-\epsilon_0}$ so that $A_{\gamma} = L^1_{m^{\epsilon}_{d}}(G)$, $\tilde{A}_{\gamma} = L^1(G)$, 
\item $A'_{\gamma} = (A'_1 + A'_2) \cap A_{\gamma} = L^1_{m^{\epsilon}_{d},0}(G)$ and $\tilde{A}_{\gamma}' = (\tilde{A}_1' + \tilde{A}_2') \cap \tilde{A}_{\gamma} = L^1_0(G)$.
\end{enumerate}
We deduce that for some explicit constant $C_{\epsilon} > 0$, for all $t > 0$, 
\[ \vertiii{S_t}_{L^1_{m^{\epsilon}_{d},0}(G) \to L^1_0(G)} = C_{\epsilon} \tilde{\Theta}(t)^{\frac{d - \epsilon}{d+1-\epsilon}} \le C_{\epsilon} \frac{1}{(1+t)^{d-\epsilon}}. \]

\vspace{-.5cm}

\end{proof}

\begin{proof}[Proof of Theorem \ref{thm:equilibrium}]
\textbf{Step 1: Uniqueness.} Assume that there exists two functions $f_{\infty}, g_{\infty}$ in $L^1_{m^{\epsilon}_d}(G)$ with the desired properties. Applying Lemma \ref{lemma:interpolation_final}, we have, for some $C > 0$ and all $t \ge 0$,
\begin{align*}
\|S_t(f_{\infty} - g_{\infty})\|_{L^1} \le C \frac{1}{(t+1)^{d-\epsilon}} \|f_{\infty}-g_{\infty}\|_{m^{\epsilon}_d}.
\end{align*}
For all $t \ge 0$, we have $S_t f_{\infty} = f_{\infty}$ and $S_t g_{\infty} = g_{\infty}$. Set $\Theta(t) = \frac{C}{(1+t)^{d-\epsilon}}$. We deduce that, for all $t \ge 0$,
\[ \|f_{\infty} - g_{\infty}\|_{L^1} \le \Theta(t) \|f_{\infty} - g_{\infty}\|_{m^{\epsilon}_d}. \]
We conclude that $f_{\infty} = g_{\infty}$ a.e. on $G$ since $\Theta(t) \to 0$ as $t \to \infty$. 

\vspace{.3cm}

\textbf{Step 2: Existence.} Let $g \in L^1_{m^{\epsilon}_{d+1}}(G)$ with $g \ge 0$ and $\langle g \rangle = 1$. We apply Lemma \ref{lemma:contraction_norm} and fix $T > T_0$ with $T_0$ given by the lemma. We set, for all $k \ge 1$, 
\[ g_k := S_{Tk} g, \qquad \hbox{and } f_k := g_{k+1} - g_k. \]
By mass conservation, for all $k \ge 1$, $\langle g_k \rangle = 1$ so that $\langle f_k \rangle = 0$ and $f_k \in L^1_{m^{\epsilon}_{d+1},0}(G)$. Applying \eqref{eq:Contraction_Lemma}, for two constants $\beta, \alpha > 0$, setting $\|\cdot\|_{\beta} = \|\cdot\|_{L^1} + \beta \|\cdot\|_{m^{\epsilon}_{d+1}}$, for all $k \ge 1$, we have
\[ \|S_{T} f_k\|_{\beta} + \alpha \|S_T f_k\|_{m^{\epsilon}_{d}} \le \|f_k\|_{\beta} + \frac{\alpha}{3} \|f_k\|_{m^{\epsilon}_{d}}. \]
We introduce the modified norm $\vertiii{\cdot}_{\tilde{\alpha}}$ defined by $\vertiii{\cdot}_{\tilde{\alpha}} = \|\cdot\|_{\beta} + \frac{\alpha}{3} \|\cdot\|_{m^{\epsilon}_{d}}$, so that the previous inequality reads 
\begin{align}
\label{eq:ineq_norm_tilde_alpha}
\vertiii{S_T f_k}_{\tilde{\alpha}} + \frac{2 \alpha}{3} \|S_T f_k\|_{m^{\epsilon}_{d}} \le \vertiii{f_k}_{\tilde{\alpha}}. 
\end{align}
This implies that, for all $k \ge 1$,
\[ \vertiii{f_{k+1}}_{\tilde{\alpha}} \le \vertiii{f_k}_{\tilde{\alpha}}, \]
hence, $(\vertiii{f_k}_{\tilde{\alpha}})_{k \ge 1}$ is non-negative, non-increasing, and is therefore a converging sequence. We fix $\delta > 0$. The previous observation implies that for $N \ge 0$ large enough and $p > l \ge N$,
\[ 0 \le \vertiii{f_l}_{\tilde{\alpha}} - \vertiii{f_p}_{\tilde{\alpha}} \le \frac{2 \alpha}{3} \delta. \]
For such $N, l$ and $p$, we thus have, using \eqref{eq:ineq_norm_tilde_alpha}
\begin{align*}
\frac{2 \alpha}{3} \big\| g_{p+1} - g_{l+1} \big\|_{m^{\epsilon}_{d}} &= \frac{2 \alpha}{3} \Big\| \sum_{k = l+1}^p f_k \Big\|_{m^{\epsilon}_{d}} \\
&\le \sum_{k=l}^{p-1} \frac{2 \alpha}{3} \big\| S_T f_k\|_{m^{\epsilon}_{d}} \\
&\le \sum_{k=l}^{p-1} \Big( \frac{2 \alpha}{3} \|S_T f_k\|_{m^{\epsilon}_{d}} + \vertiii{S_T f_k}_{\tilde{\alpha}} \Big) - \sum_{k = l}^{p-1} \vertiii{S_T f_k}_{\tilde{\alpha}} \\
&\le \sum_{k=l}^{p-1} \vertiii{f_k}_{\tilde{\alpha}} - \sum_{k=l}^{p-1} \vertiii{S_T f_k}_{\tilde{\alpha}} \\
&= \vertiii{f_l}_{\tilde{\alpha}} - \vertiii{f_p}_{\tilde{\alpha}} \le \frac{2 \alpha}{3} \delta, 
\end{align*}
by choice of $l,p$. We deduce that the sequence $(g_k)_{k \ge 1}$ is a Cauchy sequence in the Banach space $L^1_{m^{\epsilon}_{d}}(G)$, and thus converges towards a limit $f_{\infty} \in L^1_{m^{\epsilon}_{d}}(G)$ with $\langle f_{\infty} \rangle = \langle g \rangle$ by mass conservation. A similar argument to the one in Step 1 can be used to prove that this limit is independent of the choice $g \in L^1_{m^{\epsilon}_{d+1}}(G)$ with $\langle g \rangle = 1$. 
\end{proof}

\begin{proof}[Proof of Corollary \ref{coroll:cvg-to-eq}] The result follows simply by applying Lemma \ref{lemma:interpolation_final} with $g = f_{\infty}$ given by Theorem \ref{thm:equilibrium}. 
\end{proof}

\section{Steady state and velocity flow in a mixed Cercignani-Lampis model}
\label{Section:steady-flow}

In this section, we use a simplified, two-dimensional model to provide two new observations associated with two key questions regarding the free-transport problem with Cercignani-Lampis boundary condition:
\begin{enumerate}
\item the form of the associated steady state,
\item the velocity flow induced by this steady state.
\end{enumerate}
It is worth drawing a comparison here with the diffuse case, that is when $\rp \equiv 1$ and $\rt \equiv 1$. In the latter situation, there exists an explicit form for the steady state, and, quite surprisingly, it induces no velocity flow. We refer to Sone \cite[Chapter 2]{Sone_Molecular_Gas_2007} for the derivation of the explicit form, from which the absence of a velocity flow follows immediatly. 
On the other hand, the numerical study of Kosuge et al. \cite{Kosuge_steady_2011} shows that when one considers the general Cercignani-Lampis boundary condition, there is a non-zero velocity flow. In this section we consider a mixed model, that we describe now.
Let $\Omega = [0,1] \times [0,1]$, we consider the kinetic free-transport equation in $\Omega$:
\begin{align*}
\partial_t f(t,x,v) + v \cdot \nabla_x f(t,x,v) \qquad (t,x,v) \in \RR_+ \times \Omega \times \RR^2,
\end{align*}
along with a mixed boundary condition, of periodic type in the first spatial coordinate, which, for $t \ge 0$, $x_2 \in [0,1]$, $v_2 \in \RR$, takes the form
\begin{align*}
f(t,(0,x_2),(v_1,v_2)) &= f(t,(1,x_2),(v_1,v_2)), \qquad \hbox{ if } v_1 < 0, \\
f(t,(1,x_2),(v_1,v_2)) &= f(t,(0,x_2),(v_1,v_2)), \qquad \hbox{ if } v_1 > 0
\end{align*}
and of diffuse-Cercignani-Lampis type in the second spatial coordinate, by which we mean a diffuse reflection condition with temperature $T_1(x) > 0$ at $x_2 = 1$: for all $t \ge 0$, $x_1 \in [0,1]$, $v \in \RR^2$ with $v = (v_1,v_2)$ and $v_2 < 0$, 
\begin{align}
\label{eq:bdary_condition_x_2_1}
f(t,(x_1,1),v) &= \frac{e^{-\frac{|v|^2}{2T_1(x)}}}{T_1(x) \sqrt{2 \pi T_1(x)}} \int_{\{w \in \RR^2 : w_2 > 0\}} w_2 \, f(t,(x_1,1),w)   \, \d w,
\end{align}
and, at $x_2 = 0$, for $x_1 \in [0,1]$, a Cercignani-Lampis boundary condition with temperature $T_2(x) > 0$ and accommodation parameters $(\rp, \rt) \in (0,1) \times (0,2)$: for $v$ with $v_2 > 0$,
\begin{align}
\label{eq:bdary_condition_x_2_2}
f(t,x,v) &= \int_{\{w \in \RR^2 : w_2 < 0\}} (-w_2) R(w \to v; x) f(t,x,w) \, \d w  \nonumber \\
&= \int_{\{w \in \RR^2: w_2 < 0\}} (-w_2) \frac{f(t,x,w)}{T_2(x) \rp \sqrt{2 \pi T_2(x) \rt(2-\rt) }} e^{-\frac{|v_2|^2 + (1-\rp)|w_2|^2}{2 T_2(x) \rp}}  \nonumber \\
&\qquad \times I_0 \Big(\frac{(1-\rp)^{\frac12} v_2 w_2}{T_2(x) \rp} \Big) \, e^{-\frac{|v_1 - (1-\rt) w_1|^2}{2 T_2(x) \rt (2-\rt)}} \, \d w.
\end{align}

\begin{center}
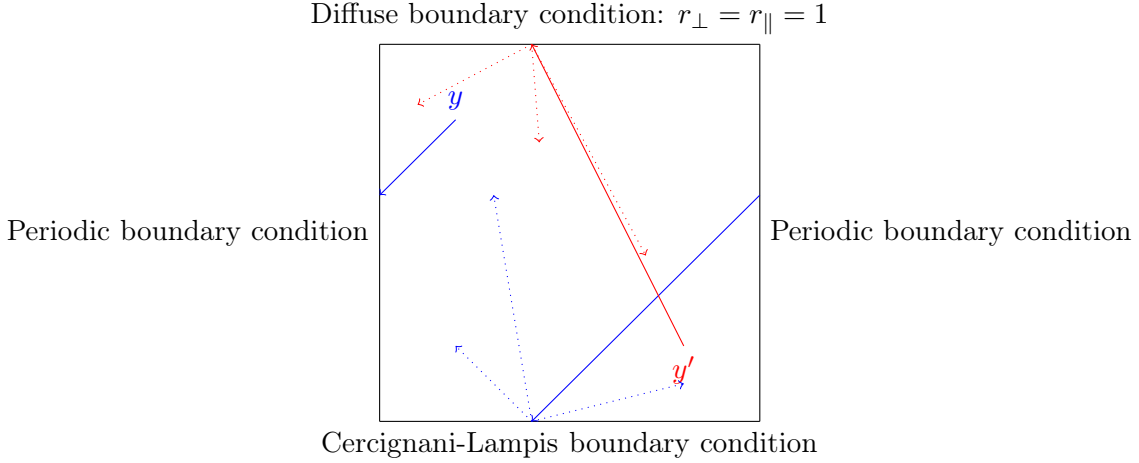
\begin{figure}
\begin{tikzpicture}
\draw (0,0) -- (5,0) ; 
\draw (0,0) -- (0,5) ;
\draw (5,0) -- (5,5) ;
\draw (0,5) -- (5,5) ;
\draw[blue, ->] (1,4) -- (0,3) ;
\draw[blue, ->] (5,3) -- (2,0) ;
\draw (1,4) node[blue, above]{$y$} ;
\draw[red, ->] (4,1) -- (2,5) ;
\draw (4,1) node[red, below]{$y'$} ;
\draw (5,2.5) node[right]{$\hbox{Periodic boundary condition}$} ;
\draw (0,2.5) node[left]{$\hbox{Periodic boundary condition}$} ;
\draw(2.5,5) node[above]{$\hbox{Diffuse boundary condition: } \rp = \rt = 1$};
\draw(2.5,0) node[below] {$\hbox{Cercignani-Lampis boundary condition}$} ;
\draw [blue, dotted, ->] (2,0) -- (1,1); 
\draw [blue, dotted, ->] (2,0) -- (4,.5);
\draw [blue, dotted, ->] (2,0) -- (1.5,3); 
\draw [red, dotted, ->] (2,5) -- (0.5, 4.2);
\draw [red, dotted, ->] (2,5) -- (3.5, 2.2);
\draw [red, dotted, ->] (2,5) -- (2.1, 3.7);
\end{tikzpicture}
\caption{A view of the toy model. In blue, the trajectory of the point starting at time $0$ from $y = (0.2,0.8)$ with velocity $v = (-0.2, -0.2)$. In red, the trajectory of the point starting at time $0$ from $y' = (0.8,0.2)$ with velocity $v' = (-0.2, 0.4)$. Dotted vectors denote possible outcoming velocities after the collisions with the boundary.}
\end{figure} 
\end{center}

We write $K_2$ for the boundary operator associated to those conditions, and rewrite the problem in the following form
\begin{align}
\label{eq:problem_2}
\left\{
\begin{array}{ll}
\partial_t f(t,x,v) + v \cdot \nabla_x f(t,x,v) = 0, \qquad &(t,x,v) \in \RR_+ \times \Omega \times \RR^2, \\
f(t,x,v) = K_2 f(t,x,v), \qquad &(t,x,v) \in \RR_+ \times \Sigma_-, \\
f(0,x,v) = f_0(x,v), \qquad &(x,v) \in G. 
\end{array}
\right.
\end{align}

In addition, we make the following hypothesis. 
\begin{hypo}
\label{hypo_Variable}
We set $T_1 \equiv 1$, $\rp \in (0,1)$, $\rt \in (0,2)$ and assume that $\rt(2-\rt) = \rp$.  We also assume that for all  $x \in [0,1] \times \{0\}$, $T_2(x) < 1$.
\end{hypo}
Hence, we fix the temperature corresponding to the diffuse boundary condition and allow the one associated with the Cercignani-Lampis boundary condition to vary. 

As for problem \eqref{eq:main_pb}, the boundary operator $K_2$ is conservative and stochastic, and thus problem \eqref{eq:problem_2} is governed by a $C_0$-stochastic semigroup $(S_t)_{t \ge 0}$, non-negative, preserving mass, and such that for all $f_0 \in L^1(G)$, for all $t \ge 0$, $S_t f_0 = f(t,\cdot)$ is the unique solution in $L^{\infty}([0,\infty);L^1(G))$ to \eqref{eq:problem_2} taken at time $t$.

In this context, using the periodicity in the $x_1$ variable, we define the function $\tilde{\sigma}$ as
\begin{align*}
\tilde{\sigma}(x,v) = 
\left\{ 
\begin{array}{ll}
\inf\{t > 0, x_2 + tv_2 = 0 \}, &\qquad \hbox{ if } x_2 \in (0,1], v_2 < 0, \\
\inf\{t > 0, x_2 + tv_2 = 1 \}, &\qquad \hbox{ if } x_2 \in [0,1), v_2 > 0, \\
0 &\qquad \hbox{ otherwise.}
\end{array}
\right.
\end{align*}
Moreover, we let $\tilde q(x,v)$ be define for all $x \in \bar{\Omega}$, $v \in \RR^2$, by
\begin{align*}
\tilde q(x,v) = (x_1 + \tilde \sigma(x,v)v_1 - [x_1 + \tilde \sigma(x,v)v_1], x_2 + \tilde \sigma(x,v)v_2),
\end{align*}
where $[y]$ denotes the integer part of $y \in \RR$, so that $ \tilde q(x,v) \in [0,1]^2$ for all $(x,v)$.
In this section, for the sake of clarity, we sometimes (abusively) write $R(u \to v; T)$ where $T > 0$ is the temperature, rather than the corresponding point. Hence, for all $x \in \pO$, $R(u \to v; x) = R(u \to v; T(x))$. Since $x \in \RR^2$ while $T(x) > 0$, any possible ambiguity can always be solved by checking the ambiant space for this variable. 

In the following theorem, we give an explicit formula for the steady state associated to the problem \eqref{eq:problem_2}, and prove that the corresponding velocity flow is zero.
\begin{thm}
\label{thm:steady}
Assume Hypothesis \ref{hypo_Variable} holds. 
Let, for all $(x,v) \in \Omega \times \RR^2$ with $x = (x_1,x_2)$, $v = (v_1,v_2)$,
\begin{align}
\label{eq:steady_state_mixed}
f_{\infty}((x_1,x_2),(v_1,v_2))  &= \beta \Big(  \mathbf{1}_{\{v_2 < 0\}} \frac{e^{-\frac{|v|^2}{2}}}{ \sqrt{2\pi}} \\
&\qquad + \mathbf{1}_{\{v_2 > 0\}} \int_{\{u \in \RR^2: u_2 < 0\}} (-u_2) R \Big(u \to v; T_2(\tilde q(x,-v)) \Big) \frac{e^{-\frac{|u|^2}{2}}}{\sqrt{2\pi}} \d u \Big) \nonumber
\end{align}
where $\beta > 0$ is chosen so that 
\[ \int_{\Omega \times \RR^2} f_{\infty}(x,v) \, \d x \, \d v = 1. \]
Then $f_{\infty}$ is a steady state for the problem \eqref{eq:problem_2}.
Moreover, for all $x \in \Omega$, 
\begin{align}
\label{eq:steady_flow}
\int_{\RR^2} v f_{\infty}(x,v) \d v = 0. 
\end{align}
\end{thm}
%

Before getting to the proof, we prove the following lemma, adapted from Chen \cite{Chen_CL_2020}. 

\begin{lemma}
\label{lemma:chen}
For any $a > 0$, $b > 0$ with $a < b$, $w \in \RR$, 
\begin{align}
\label{eq:lemma_chen_gaussian}
\sqrt{\frac{b}{\pi}} \int_{\RR} e^{av^2} e^{-b(v-w)^2} \, \d v = \sqrt{\frac{b}{b-a}}e^{\frac{ab}{b-a} w^2},
\end{align}
and
\begin{align}
\label{eq:lemma_chen_rice}
2b \int_0^{\infty} v e^{av^2} e^{-bv^2} e^{-bw^2} I_0(2bvw) \, \d v = \frac{b}{b-a} e^{\frac{ab}{b-a}w^2}.
\end{align}
Therefore, for all $v = (v_1,v_2) \in \RR^2$ with $v_2 > 0$, for all $T_2 \in (0,1)$ (possibly depending on $v$), under Hypothesis \ref{hypo_Variable},
\begin{align}
\label{eq:preliminary_beta}
\int_{\{ u \in \RR^2, u_2 < 0\}} (-u_2) R(u \to v, T_2) \frac{e^{-\frac{|u|^2}{2}}}{\sqrt{2 \pi}} \, \d u = \frac{e^{-\frac{|v|^2}{2(1 - \rp + T_2 \rp)}}}{(1-\rp + T_2 \rp)^{\frac{3}{2}} \sqrt{2\pi}} . 
\end{align}
\end{lemma}

\begin{proof}
Equation \eqref{eq:lemma_chen_gaussian} is a straightforward adaptation in dimension one of the computation done in \cite[Lemma 11]{Chen_CL_2020}. Equation \eqref{eq:lemma_chen_rice} is given in \cite[Lemma 12]{Chen_CL_2020}.

We now turn to the proof of \eqref{eq:preliminary_beta}. We recall from Chen \cite{Chen_CL_2020} that we have the reciprocity property: for all $a = (a_1,a_2) \in \RR^2$ with $a_2 > 0$, $b = (b_1,b_2) \in \RR^2$ with $b_2 < 0$, for all $T > 0$, we have
\begin{align*}
R(b \to a; T) = R(-a \to -b; T) e^{-\frac{a^2}{2T}} e^{\frac{b^2}{2T}}. 
\end{align*}
Applying this inside the integral, and performing the change of variable $u \to -u$, we find
\begin{align*}
\int_{\{ u \in \RR^2, u_2 < 0\}} \hspace{-.5cm} (-u_2) R(u \to v, T_2) \frac{e^{-\frac{|u|^2}{2}}}{\sqrt{2 \pi}} \, \d u = \frac{e^{-\frac{|v|^2}{2T_2}}}{\sqrt{2\pi}} \int_{\{u \in \RR^2, u_2 > 0\}} \hspace{-.5cm} u_2 R(-v \to u ; T_2) e^{-\frac{|u|^2}{2}} e^{\frac{|u|^2}{2T_2}} \, \d u. 
\end{align*}
The integral on the right-hand side writes (recall that $\rp = \rt(2-\rt)$ by hypothesis and that $I_0$ is even)
\begin{align*}
\int_{\{u \in \RR^2, u_2 > 0\}} u_2 &R(-v \to u ; T_2) \frac{e^{-\frac{|u|^2}{2}}}{\sqrt{2 \pi}} e^{\frac{|u|^2}{2T_2}} \, \d u \\
&= \Big( \frac{1}{T_2 \rp} \int_0^{\infty} u_2 e^{-\frac{u_2^2}{2}} e^{\frac{u_2^2}{2T_2}} e^{-\frac{u_2^2}{2 T_2 \rp}} e^{-\frac{v_2^2 (1-\rp)}{2 T_2 \rp}} I_0\Big(\frac{(1-\rp)^{\frac12} u_2 v_2}{T_2 \rp} \Big) \, \d u_2 \Big) \\
&\qquad \times \Big( \frac{1}{\sqrt{2 \pi \rp T_2}} \int_\RR e^{-\frac{u_1^2}{2}} e^{\frac{u_1^2}{2T_2}} e^{-\frac{(u_1 + (1-\rt) v_1)^2}{2 T_2 \rp}} \d u_1 \Big).
\end{align*}
We apply first \eqref{eq:lemma_chen_rice} with $b = \frac{1}{2 T_2 \rp}$, $w = \sqrt{1-\rp} v_2$, $a = (\frac{1}{2T_2} - \frac{1}2) \in (0,b)$ since $\rp \in (0,1)$ and $T_2 < 1$, and we find
\begin{align*}
\Big( \frac{1}{T_2 \rp} \int_0^{\infty} &u_2 e^{-\frac{u_2^2}{2}} e^{\frac{u_2^2}{2T_2}} e^{-\frac{u_2^2}{2 T_2 \rp}} e^{-\frac{v_2^2 (1-\rp)}{2 T_2 \rp}} I_0\Big(\frac{(1-\rp)^{\frac12} u_2 v_2}{T_2 \rp} \Big) \, \d u_2 \Big) \\
&= \frac{1}{1 - \rp + T_2 \rp} e^{v_2^2(\frac{1}{2T_2} - \frac1{2(1-\rp + T_2 \rp)})}.
\end{align*}
We now apply \eqref{eq:lemma_chen_gaussian} with $b = \frac{1}{2T_2\rp}$, $a = \frac{1}{2T_2} - \frac12 \in (0,b)$ and $w = -(1-\rt) v_1$:
\begin{align*}
 \frac{1}{\sqrt{2 \pi \rp T_2}} \int_\RR e^{-\frac{u_1^2}{2}} e^{\frac{u_1^2}{2T_2}} e^{-\frac{(u_1 + (1-\rt) v_1)^2}{2 T_2 \rp}} \d u_1 =\frac{ \sqrt{1 - \rp + T_2 \rp}}{1 - \rp + T_2 \rp} e^{v_1^2(\frac{1}{2T_2} - \frac1{2(1-\rp + T_2 \rp)})},
\end{align*} 
where we used that $(1-\rt)^2 = 1 - \rt(2-\rt) = 1 - \rp$. The conclusion follows by bringing together both terms. Note that the derivation can be performed in the same manner if $T_2$ depends on $v$.  
\end{proof}

\begin{proof}[Proof of Theorem \ref{thm:steady}]

\textbf{Step 1 : steady state of the free-transport equation without boundary condition.} 

Note that for all $x \in \Omega$, $v \in \RR^d$, for all $h > 0$ small enough, $\tilde q(x+hv,-v) = \tilde q(x,-v)$. Hence, $v \cdot \nabla_x f_{\infty}(x,v) = 0$ and since it does not depend on $t$, this shows that the candidate is a solution to the free-transport equation without boundary conditions. 

\vspace{.3cm}
 
We only need to check that the boundary conditions are satisfied. We clearly have that the boundary conditions at $x_1 = 0$ and $x_1 = 1$ are satisfied using the definition of $f_{\infty}$ and $\tilde{q}$.   We now turn to the to the boundary conditions at $x_2 = 0$ and $x_2 = 1$. 

\vspace{.3cm}

\textbf{Step 2: boundary condition at $x_2 = 1$.}
Let us compute the left-hand side of \eqref{eq:bdary_condition_x_2_1} and show that $f_{\infty}$ indeed satisfies the boundary condition. The former writes, for $v = (v_1,v_2) \in \RR^2$ with $v_2 < 0$ and for all $ x = (x_1,1)$ with $x_1 \in [0,1]$,
\begin{align*}
\frac{e^{-\frac{|v|^2}{2}}}{ \sqrt{2 \pi}} \int_{\{w \in \RR^2:w_2 > 0\}} w_2 \,  f_{\infty}(x,w)\,  \d w
\end{align*}
and we only need to prove that  
\[  \int_{\{w \in \RR^2:w_2 > 0\}} w_2 \,  f_{\infty}(x,w)\,  \d w = \beta  \]
to conclude this step. For this, we will use Lemma \ref{lemma:chen}. Indeed

\begin{align*}
\int_{\{w \in \RR^2:w_2 > 0\}} &w_2 \,  f_{\infty}(x,w)\,  \d w \\
&= \beta \int_{\{w \in \RR^2:w_2 < 0\}} \hspace{-.3cm} (-w_2) \,  \int_{\{u \in \RR^2: u_2 < 0\}} \hspace{-.3cm} (-u_2) R(u \to -w; T_2(\tilde q(x,w)) \frac{e^{-\frac{|u|^2}{2}}}{\sqrt{2\pi}} \, \d u \,  \d w
\end{align*}
where we performed the change of variable $w \to -w$. Applying \eqref{eq:preliminary_beta}, we find
\begin{align*}
\int_{\{w \in \RR^2:w_2 > 0\}} w_2 \,  f_{\infty}(x,w)\,  \d w = \beta \int_{\{w \in \RR^2:w_2 < 0\}} (-w_2) \frac{e^{-\frac{|w|^2}{2(1-\rp + T_2(\tilde q(x,w)) \rp}}}{(1-\rp + T_2(\tilde q(x,w)) \rp)^{\tfrac32} \sqrt{2\pi}} \, \d w.
\end{align*}
We now write $w$ in polar coordinates, with $s = |w|$ and $\theta$ the corresponding angle with the vector $e_1 = (0,1)$, with the condition $\theta \in (-\pi,0)$ to ensure that $w_2 < 0$. Note that $\tilde q(x,w) = \tilde q(x,\underline{u}(\theta))$, where $\underline{u}(\theta)$ is the unit vector associated to the angle $\theta$, thus this quantity is independent of $s$. Recall that the Jacobian of this polar change of coordinates is simply given by $s$. Hence the integral rewrites
\begin{align*}
\int_{\{w \in \RR^2:w_2 > 0\}} \hspace{-.2cm} w_2 \,  f_{\infty}(x,w)\,  \d w  = \beta  \int_{-\pi}^0 \int_{0}^{\infty} (-s^2 \sin(\theta)) \frac{e^{-\frac{s^2}{2(1-\rp + T_2(\tilde q(x,\underline{u}(\theta))) \rp}}}{(1-\rp + T_2(\tilde q(x,\underline{u}(\theta)) \rp)^{\tfrac32} \sqrt{2\pi}} \, \d s \d \theta
\end{align*}
We now perform the change of variable $s \to \frac{s}{\sqrt{1-\rp + T_2(q(x,\underline{u}(\theta))) \rp}}$ (note that the denominator is independent of $s$) in the integral on $s$, to find
\begin{align*}
\int_{\{w \in \RR^2:w_2 > 0\}} w_2 \,  f_{\infty}(x,w)\,  \d w = \beta \int_{-\pi}^0 \int_{0}^{\infty} (-s^2 \sin(\theta)) \frac{e^{-\frac{s^2}{2}}}{\sqrt{2\pi}} \, \d s \d \theta
\end{align*}
and applying the reverse change of coordinates $(s,\theta) \to w$ from $\RR_+ \times (-\pi,0)$ to the set $\{w \in \RR^2, w_2 < 0\}$, we obtain
\begin{align*}
\int_{\{w \in \RR^2:w_2 > 0\}} w_2 \,  f_{\infty}(x,w)\,  \d w & = \beta \int_{\{w \in \RR^2:w_2 < 0\}} (-w_2) \frac{e^{-\frac{|w|^2}{2}}}{\sqrt{2\pi}} \d w \\
&= \beta \Big( \int_{-\infty}^0 (-w_2) e^{-\frac{w_2^2}{2}} \d w_2 \Big) \Big( \int_{\RR} \frac{e^{-\frac{w_1^2}{2}}}{\sqrt{2 \pi}} \Big) = \beta.
\end{align*}

\vspace{.3cm}

\textbf{Step 3: boundary condition at $x_2 = 0$.} 
We have, for all $x = (x_1,0)$ with $x_1 \in [0,1]$, for all $v = (v_1,v_2) \in \RR^2$ with $v_2 > 0$, computing the right-hand-side of \eqref{eq:bdary_condition_x_2_2},
\begin{align*}
\int_{\{u \in \RR^2 : u_2 < 0\}} &(-u_2) R(u \to v; T_2(x)) f_{\infty}(x,u) \, \d u \\
&= \beta \int_{\{u \in \RR^2 : u_2 < 0\}} (-u_2) R(u \to v; T_2(x)) \frac{e^{-\frac{|u|^2}{2}}}{\sqrt{2\pi}} \, \d u = f_{\infty}(x,v),
\end{align*}
using that, since $x_2 = 0$ and $v_2 > 0$, $\tilde q(x,-v) = x$ and the formula \eqref{eq:steady_state_mixed}.

\vspace{.3cm}

The next two steps are devoted to the proof of \eqref{eq:steady_flow}.

\vspace{.3cm}

\textbf{Step 4: flow for $v_1$.}
Let $x \in \Omega$. 
Clearly,
\begin{align*}
\int_{\{v \in \RR^2: v_2 < 0\}} v_1 f_{\infty}(x,v) \d v = \beta \int_{\{v \in \RR^2: v_2 < 0\}} v_1 \frac{e^{-\frac{|v|^2}{2} }}{\sqrt{2 \pi}} \, \d v = 0,
\end{align*}
by oddity. On the other hand, by applying \eqref{eq:preliminary_beta},
\begin{align*}
\int_{\{v \in \RR^2: v_2 >  0\}} &v_1 f_{\infty}(x,v) \d v \\
&= \beta \int_{\{v \in \RR^2: v_2 >  0\}} v_1 \int_{\{u \in \RR^2: u_2 < 0\}} (-u_2) R\Big(u \to v; T_2(\tilde q(x,-v))\Big) \frac{e^{-\frac{|u|^2}{2}}}{\sqrt{2\pi}} \d u \d v\\
 &= \beta \int_{\{v \in \RR^2: v_2 >  0\}} v_1 \frac{e^{-\frac{|v|^2}{2(1-\rp + T_2(\tilde q(x,-v)) \rp)}}}{(1-\rp + T_2(\tilde q(x,-v)) \rp)^{\tfrac{3}{2}} \sqrt{2 \pi} } \d v 
\end{align*}
and by applying again the change in polar coordinates and the change of variables from Step 2, we find
\begin{align*}
\int_{\{v \in \RR^2: v_2 >  0\}} v_1 \frac{e^{-\frac{|v|^2}{2(1-\rp + T_2(q(x,-v)) \rp)}}}{(1-\rp + T_2(q(x,-v)) \rp)^{\tfrac{3}{2}} \sqrt{2 \pi} } \d v 
&= \int_{\{v \in \RR^2, v_2 > 0\}} v_1 \frac{e^{-\frac{|v|^2}{2}}}{\sqrt{2\pi}} \d v = 0 
\end{align*}
by oddity.

\vspace{.3cm}

\textbf{Step 5: flow for $v_2$.} Let again $x \in \Omega$. On the one hand,
\begin{align*}
\int_{\{v \in \RR^2:v_2 < 0\}} v_2 f_{\infty}(x,v) \d v &= \beta \int_{\{v \in \RR^2:v_2 < 0\}} v_2 \frac{e^{-\frac{|v|^2}{2 }}}{\sqrt{2 \pi}} \, \d v \\
&= \beta \Big( \int_\RR \frac{e^{-\frac{v_1^2}{2}}}{\sqrt{2 \pi}} \d v_1 \Big) \Big( \int_{-\infty}^0 v_2 e^{-\frac{|v_2|^2}{2}} \, \d v_2 \Big) = -\beta,
\end{align*}
by a simple decomposition. On the other hand
\begin{align*}
\int_{\{v \in \RR^2: v_2 > 0\}} &v_2 f_{\infty}(x,v) \d v \\
&= \beta \int_{\{ v \in \RR^2, v_2 > 0\}} v_2 \int_{\{u \in \RR^2: u_2 < 0\}} (-u_2) R\Big(u \to v; T_2( \tilde q(x,-v)) \Big) \frac{e^{-\frac{|u|^2}{2}}}{\sqrt{2 \pi}} \d u \d v. 
\end{align*}
The double integral on the right-hand side is exactly the one computed in Step 2 (note that here, $x \in \Omega$ rather than $x \in \pO$, but the same computations apply), and is thus worth $1$. Therefore
\begin{align*}
\int_{\{v \in \RR^2: v_2 > 0\}} &v_2 f_{\infty}(x,v) \d v = \beta,
\end{align*}
which concludes the proof.
\end{proof}

\begin{rmk}
The result of Theorem \ref{thm:steady} is not surprising in the case where $\rp = \rt = 1$ even at $x_2 = 0$, since this corresponds to the diffuse boundary conditions at both boundaries for $v_2$. Thus, we expect the absence of steady flow from the result of Sone \cite{Sone_Molecular_Gas_2007} in this case.
\end{rmk}

\begin{rmk}
More interestingly, for $(\rp,\rt) \ne (1,1)$, Theorem \ref{thm:steady} shows that the interaction between a diffuse boundary condition and a ``real'' (i.e. not diffuse) Cercignani-Lampis condition is not enough to generate a velocity flow. The idea is that the diffuse boundary condition kills all correlations with the past. This can be seen in the second computation for the flow for $v_1$: the fact that the last integral in $u_1$ is $0$ is the key point. We plan to investigate this model with two ``real'' Cercignani-Lampis boundary conditions by means of a probabilistic approach in the near future.
\end{rmk}

\bibliographystyle{plain}
\bibliography{biblioCL}

\end{document}